\documentclass[11pt]{article}

\usepackage{bm}
\usepackage[top=1in,bottom=1in,left=1in,right=1in]{geometry}

\usepackage{textgreek,mathtools}
\usepackage{amsmath}
\usepackage{graphicx,subfig}
\usepackage[colorinlistoftodos]{todonotes}
\usepackage[colorlinks=true, allcolors=blue]{hyperref}
\usepackage{amsmath}
\usepackage{upgreek}
\usepackage{amssymb}
\usepackage{amsfonts}
\usepackage{amsthm}
\usepackage{mathrsfs}
\usepackage{fontenc}
\usepackage{empheq}
\usepackage{enumerate}
\usepackage{comment}
\usepackage{appendix}
\usepackage{bbm}
\usepackage{algorithm,algpseudocode}
\mathtoolsset{showonlyrefs}

\theoremstyle{plain}
\newtheorem{thm}{Theorem}[section]

\newtheorem{lem}[thm]{Lemma}

\newtheorem{assu}[thm]{Assumption}

\theoremstyle{definition}
\newtheorem{defn}[thm]{Definition}

\theoremstyle{remark}
\newtheorem{rem}[thm]{Remark}

\newcommand{\RR}{\mathbb{R}}
\newcommand{\EE}{\mathbb{E}}

\newcommand{\PP}{\mathbb{P}}

\newcommand{\MCL}{\mathcal{L}}
\newcommand{\MCP}{\mathcal{P}}
\newcommand{\ud}{\,\mathrm{d}}
\newcommand{\mc}[1]{\mathcal{#1}}
\newcommand{\mb}[1]{\mathbb{#1}}
\newcommand{\half}{\frac{1}{2}}

\newcommand{\bR}{\mathbb {R}}
\newcommand{\transpose}{^{\operatorname{T}}}
\newcommand{\rmd}{\,\mathrm{d}}
\newcommand{\rmD}{\mathrm{D}}

\newcommand{\inv}{^{\raisebox{.2ex}{$\scriptscriptstyle-1$}}}

\title{Learning High-Dimensional McKean-Vlasov Forward-Backward Stochastic Differential Equations with General Distribution Dependence}

\author{Jiequn Han\thanks{Center for Computational Mathematics, Flatiron Institute, New York, NY 10010, \em{jiequnhan@gmail.com}.} \and Ruimeng Hu\thanks{Department of Mathematics, and Department of Statistics and Applied Probability, University of California, Santa Barbara, CA 93106-3080, {\em rhu@ucsb.edu}.} \and  Jihao Long\thanks{The Program in Applied and Computational Mathematics, Princeton University, Princeton, NJ 08544-1000, \em{jihaol@princeton.edu}.}}
\date{\today}

\begin{document}

\maketitle

\begin{abstract}
One of the core problems in mean-field control and mean-field games is to solve the corresponding McKean-Vlasov forward-backward stochastic differential equations (MV-FBSDEs). Most existing methods are tailored to special cases in which the mean-field interaction only depends on expectation or other moments and thus is inadequate to solve problems when the mean-field interaction has full distribution dependence.

In this paper, we propose a novel deep learning method for computing MV-FBSDEs with a general form of mean-field interactions. Specifically, built on fictitious play, we recast the problem into repeatedly solving standard FBSDEs with explicit coefficient functions. These coefficient functions are used to approximate the MV-FBSDEs' model coefficients with full distribution dependence, and are updated by solving another supervising learning problem using training data simulated from the last iteration's FBSDE solutions. We use deep neural networks to solve standard BSDEs and approximate coefficient functions in order to solve high-dimensional MV-FBSDEs.
Under proper assumptions on the learned functions, we prove that the convergence of the proposed method is free of the curse of dimensionality (CoD) by using a class of integral probability metrics (IPMs) previously developed in [Han, Hu and Long, arXiv:2104.12036]. The proved theorem shows the advantage of the method in high dimensions. We present the numerical performance in high-dimensional MV-FBSDE problems, including a mean-field game example of the well-known Cucker-Smale model whose cost depends on the full distribution of the forward process. 

\end{abstract}

\noindent\textbf{Keywords:}
	McKean-Vlasov FBSDE, fictitious play, mean-field games, Deep BSDE, maximum mean discrepancy, convergence analysis.

\section{Introduction}

The McKean-Vlasov forward-backward stochastic differential equations (MV-FBSDEs) arise as natural formulations for mean-field games (MFGs) and mean-field control (MFC) problems. Initially introduced by Lasry-Lions \cite{LaLi1:2006,LaLi2:2006, LaLi:2007} and Huang-Malhame-Caines \cite{HuMaCa:06,HuCaMa:07}, MFG studies the strategic decision-making by a continuum of homogeneous agents where each of them aims to maximize her/his own goal. Since each agent is infinitesimal, the law of states processes is considered fixed while an agent optimizes her cost functional. Therefore, it is merely a standard optimal control problem plus a fixed point argument (that the distribution of the solution is equal to the distribution one starts from). From a slightly different perspective, MFC analyzes the optimal control of stochastic differential equations (SDEs) of the McKean-Vlasov type, which can be interpreted as the limiting regime of a cooperative game among large population agents. This objective adds an extra optimization layer to the fixed point, that is, the law of state processes changes as one optimizes her goal. Despite their differences, both problems can be solved via analyzing the MV-FBSDEs using the Pontryagin maximum principle \cite{carmona2013probabilistic,CaDe1:17}, and both solutions serve as approximate equilibrium strategies for large populations of individuals whose interactions and objective functions are of the mean-field type. To further explain their differences and relations, we refer readers to \cite[Section~6]{CaDe1:17}.

In this paper, we are interested in developing efficient deep learning methods and analyzing the numerical convergence for the following MV-FBSDEs,
\begin{equation}\label{def:MV-FBSDE}
\hspace{-5pt}\begin{dcases}
  \mathrm{d} X_t = b(t, X_t, Y_t, Z_t, \MCL(X_t, Y_t, Z_t)) \ud t + \sigma(t, X_t, Y_t, Z_t, \MCL(X_t, Y_t, Z_t)) \ud W_t, \, X_0 = x_0,  \\
  \mathrm{d} Y_t = -h(t, X_t, Y_t, Z_t, \MCL(X_t, Y_t, Z_t)) \ud t + Z_t \ud W_t, \quad Y_T = g(X_T, \MCL(X_T)),
\end{dcases}
\end{equation}
on a finite horizon $[0,T]$, where $\MCL(\cdot)$ denotes the marginal law of a process, and $(b, \sigma, h, g)$ are measurable functions of compatible dimensions, with specific forms detailed in Section~\ref{sec:algorithm}. Particularly, we shall focus on computing MV-FBSDEs in high dimensions with a general form of mean-field interactions, while most existing methods only deal with special cases when the mean-field interaction is described via expectations or other moments.

Theoretically, existence and uniqueness results for MV-FBSDEs have been recently developed \cite{carmona2013mean,carmona2013probabilistic,carmona2015forward}, mostly tackled by a compactness argument and fixed point theorems. In particular, when $\sigma$ is free of $Z_t$ and $\mc{L}(Z_t)$, under suitable conditions \cite[Theorem~4.29, Remark~4.30]{CaDe1:17}, the MV-FBSDEs \eqref{def:MV-FBSDE} admit a solution with a decoupling field
\begin{equation}\label{def:decouplingfield}
    Y_t = u(t, X_t), \quad Z_t = v(t, X_t) 
\end{equation}
where $v(t, x) = \partial_x u(t, x) \sigma(t, x, u(t, x), (I_d, u(t, \cdot))(\mc{L}(X_t)))$. Several numerical algorithms have been designed for solving MV-FBSDEs, including a recursive local Picard iteration method \cite{chassagneux2019numerical}, cubature based algorithm for the decoupled setting \cite{de2015cubature}, and deep learning methods  \cite{carmona2019convergence,germain2019numerical}. Most of the work, with or without analysis, only present numerical examples of a special type of mean-field interaction, that is, $(b, \sigma, h, g)$ depends on $\mc{L}$ only through its expectation. On a related topic, MFGs and MFC problems have also been solved using partial differential equations and using neural networks for high-dimensional problems; for instance, see \cite{achdou2020mean,ruthotto2020machine,reisinger2021fast,achdou2022income}.

The work most closely related to ours is \cite{germain2019numerical}. It solves the MV-FBSDEs using deep learning, and the dependence of the $(b, \sigma, h, g)$ with respect to $\mc{L}(X_t, Y_t, Z_t)$ are solely in terms of expectations of the form $(\EE[\varphi_1(X_t)], \EE[\varphi_2(Y_t)], \EE[\varphi_3(Z_t)])$ with some continuous functions $\varphi_i$. To do so, one is essentially estimating some equilibrium numbers $\int_y \varphi_i(y) \,\mu(\mathrm{d}y)$ where $\mu$ is the law of the MV-FBSDE solution.  Beyond mere dependency on moments, which is the scope of this paper, the goal is to learn equilibrium functions that depend on the distribution in a nonlinear way, e.g., functions in $(t,x)$ of the form $\int_y \varphi(t, x, y) \, \mu(\mathrm{d}y)$. This presents greater challenges in terms of both algorithm design and the method employed for representing these quantities.

Here we propose a novel deep learning method to solve high-dimensional MV-FBSDEs, leveraging the idea of fictitious play \cite{brown1949some,brown1951iterative} and 
the existing machine learning BSDE solvers \cite{weinan2017deep,han2018solving,han2020convergence1,hure2020deep}, with the capability of dealing with general mean-field dependence (not only moments). The main contributions are:
\begin{enumerate}
    \item 
    Motivated by the fixed-point nature of MV-FBSDEs, we use fictitious play and solve
    \begin{equation}\label{def:MV-FBSDE-fp}
\begin{dcases}
  \ud X_t^k = b(t, \Theta_t^k, \MCL(\Theta_t^{k-1})) \ud t + \sigma(t, \Theta_t^k, \MCL(\Theta_t^{k-1})) \ud W_t, \; X_0 = x_0,  \\
  \ud Y_t^k = -h(t, \Theta_t^{k}, \MCL(\Theta_t^{k-1})) \ud t + Z_t^k \ud W_t, \; Y_T = g(X_T^k, \MCL(X_T^{k-1})),
\end{dcases}
\end{equation}
among solutions in the form of decoupling fields (that is, search for some deterministic functions $u^k, v^k$ which are viewed as approximations to $u, v$ defined in \eqref{def:decouplingfield} such that
$Y_t^k = u^k(t, X_t^k), Z_t^k = v^k(t, X_t^k)$), iteratively via existing BSDE solvers, where $\Theta^k = (X^k, Y^k, Z^k)$ is the short notation of the solution to \eqref{def:MV-FBSDE-fp}, and is understood as the $k^{th}$ iteration approximation of $(X, Y, Z)$ in \eqref{def:MV-FBSDE}. To speed up the simulation of $(X^k, Y^k)$ for the BSDE solver, we propose to use neural networks to learn the maps $(t, X_t, Y_t, Z_t) \mapsto (b, \sigma, h, g)$ by supervised learning when they depend fully on any distribution. Numerically, the function values (also known as the ``labels'') are approximated via replacing $\mc{L}(\cdot)$ by the corresponding empirical distribution of the solution from the last iteration and the learned maps are updated after each iteration. 
    
    \item We prove the convergence for the proposed algorithm. Using the integral probability metrics proposed in \cite{han2021class}, the difference between the learned process and the solution can be sufficiently small and free of the curse of dimensionality, 
    subject to sufficient smoothness of the learned functions, sufficient iterations of fictitious play, and sufficiently small time steps. 
    
    \item We introduce a mean-field game of the Cucker-Smale model where the cost depends on the full distribution of the forward process. We provide numerical benchmarks and hope this will facilitate the further study of numerical algorithms for MV-FBSDEs. 
\end{enumerate}

The rest of the paper is organized as follows. In Section~\ref{sec:algorithm}, we review the idea of fictitious play and two existing deep learning (DL)-based BSDE solvers: Deep BSDE and DBDP, and describe our proposed algorithm. Section~\ref{sec:theory} provides the convergence analysis for the proposed algorithm. Two numerical examples are presented in Section~\ref{sec:numerics}. The first one is ten-dimensional with an analytic solution to benchmark. It is worth mentioning that the first example is not the frequently used MV-FBSDEs associated with linear quadratic problems that have been extensively benchmarked in the literature. In the second example, we study a mean-field flocking problem under the Cucker-Smale model. The mean-field interaction there involves the whole distribution of the state processes. We conclude in Section~\ref{sec:conclusion}.

\medskip

\noindent\textbf{Notations.} 
Throughout the paper, we fix a probability space $(\Omega, \mc{F}, \PP)$ that supports $q$-dimensional Brownian motions $(W_t^{n, k})_{t\in [0,T]}$, $n = 1, \ldots, N$, $k = 1, \ldots$ and $(W_t)_{t\in [0,T]}$ on the fixed horizon $[0,T]$. Let $\mathbb{F} := (\mc{F}_t)_{t \in [0,T]}$ be the natural filtration  associated with $W$. We denote by $\mathbb{H}^2(\RR^d)$ the space
\begin{equation*}
    \mathbb{H}^2(\RR^d) = \left\{x \text{ is } \RR^d \text{-valued progressively measurable: } \EE\int_0^T \|x_t\|^2 \ud s < \infty \right\},
\end{equation*}
where $\|\cdot\|$ denote the $\ell^2$ norm. The space $\mathbb{S}^2(\RR^d)$ denotes all the continuous $\RR^d$-valued progressively measurable processes $x_t$ such that $\EE[\sup_{0 \leq t \leq T}\|x_t\|^2]<\infty$.
We fix the dimension of $(X_t, Y_t, Z_t)$ in \eqref{def:MV-FBSDE} by $X_t \in \RR^d$, $Y_t \in \RR^p$ and $Z_t \in \RR^{p \times q}$, and denote by $\nu_t$ and $ \mu_t$ the empirical measures of $(X_t, Y_t, Z_t)$ and $X_t$, respectively.

Denote by $\MCP(\RR^d)$ the space of probability measures on $\RR^d$, and by $\MCP^s(\RR^d)$ a subspace of $\MCP(\RR^d)$ with finite $s^{th}$-moments, \emph{i.e.}, $\mu \in \MCP^s(\RR^d)$ if and only if
\begin{equation*}
    M_s(\mu) := \left(\int_{\RR^d}\|x\|^s \ud \mu(x)\right)^{1/s} < +\infty.
\end{equation*}
 We will mostly work with probability measures with finite second moments, \emph{i.e.}, $\MCP^2(\RR^d)$.

\section{Deep Learning Algorithm for McKean-Vlasov FBSDEs}\label{sec:algorithm}

Compared to the standard FBSDEs problem, solving the MV-FBSDEs \eqref{def:MV-FBSDE} faces two additional difficulties: 
\begin{enumerate}
    \item The law $\mc{L}(X_t, Y_t, Z_t)$ is unknown a priori, but determined as a fixed-point;
    \item Even given the law $\mc{L}(X_t, Y_t, Z_t)$, the dependence of coefficient functions $(b,\sigma,h,g)$ on $\mc{L}(X_t, Y_t, Z_t)$ can be so complicated that the resulting FBSDEs are computationally challenging to solve.
\end{enumerate}
To overcome the first difficulty, we use the idea of fictitious play. Fictitious play is originally introduced by Brown \cite{brown1949some,brown1951iterative} as a learning process for finding Nash equilibrium in static games,  and has attracted many attentions and has been used in machine learning algorithms and theory for various settings; see, for instance, in \cite{Hu2:19,HaHu:19,han2020convergence,EliePerolatLauriereGeistPietquin-2019_AFP-MFG,xie2020provable,Xuanetal:21,min2021signatured,han2021deepham}. Such problems usually require finding an ``equilibrium'' quantity $Q^\ast$. In turn, one could start with an initial guess $Q^{(0)}$ of this ``equilibrium'', use it to solve the problem, and update the guess $Q^{(1)}$ accordingly. Then the problem with a given $Q^{(k)}$ is solved repeatedly, producing a sequence $Q^{(0)}, Q^{(1)}, \ldots, Q^{(k)}, \ldots$ which one hopes to have a limit $Q^\ast$. In $N$-player games, the quantity is the Nash equilibrium strategy, while in mean-field games, it is the law of the optimal state processes. In this paper, it will be the flow of measures of $(X_t, Y_t, Z_t)$. Note that the idea of fictitious play is very much in line with the theoretical construction of the solution to \eqref{def:MV-FBSDE}, which is usually done by recasting it into a well-posed fixed point problem over the argument $\mc{L}(X_t, Y_t, Z_t)$. There, the first step is to use some fixed distribution as an input and then solve it as a standard FBSDE, and the goal is to find a fixed point for the map from the input distribution to the law of the standard FBSDE in appropriate spaces of functions and measures. To overcome the second difficulty, we observe that $\mc{L}(X_t, Y_t, Z_t)$ are deterministic functions of time $t$, so we can view $m_1$ and $m_2$ in \eqref{def:MV-FBSDE-simple} as functions of $(t, x)$, and similarly view $m_3$ as a function of $x$. Therefore, we use supervised learning to learn these maps directly, given the latest empirical estimate of $\mc{L}(X_t, Y_t, Z_t)$.

Now we are ready to present the proposed algorithm. To ease the notation, we define the processes $\Theta \stackrel{\triangle}{=} (X, Y, Z)$ 
where $(X, Y, Z)$ are in \eqref{def:MV-FBSDE}, and define $\RR^\theta  \stackrel{\triangle}{=} \RR^d \times \RR^p \times \RR^{p\times q}$ which the space $\Theta_t$ lies in. In the sequel, any sub/superscript added to $\Theta$ are automatically applied to $(X, Y, Z)$, {\it e.g.},  $\Theta^k_{t_i} = (X^k_{t_i}, Y^k_{t_i}, Z^k_{t_i})$ represents the time $t_i$ value of the MV-FBSDEs solved at the $k^{th}$ iteration of fictitious play.

\subsection{The Deep MV-FBSDE Algorithm}
\label{sec:algorithm_outline}
We first assume the following structures for functions $(b, \sigma, h, g)$. Let $m_1, m_2: [0, T]\times \RR^d \times \MCP^2(\RR^\theta)\rightarrow \RR^l$ and $m_3: \RR^d \times \MCP^2(\RR^d)\rightarrow \RR^l$ be vector-valued functions such that 
\begin{equation}\label{assump:algo2}
    \begin{aligned}
       &b = b(t, \Theta_t, m_1(t, X_t, \MCL(\Theta_t))), && \sigma = \sigma(t, X_t),   \\
     &h = h(t, \Theta_t, m_2(t, X_t, \MCL(\Theta_t))), &&g = g(X_T, m_3(X_T, \MCL(X_T))). 
    \end{aligned}
\end{equation}
As introduced above, our algorithm is iterative over the solution $(X_t, Y_t, Z_t)$ and their empirical measures. We use superscripts $k=1,2,\cdots$ to index such iterations. In each iteration stage, the algorithm performs the following three steps.

\paragraph{Step~1. Update the approximation to distributions $\mc{L}(\Theta)$\protect\footnote{The iteration can also proceed with a different order of three steps, as long as some proper initial conditions corresponding to $k=0$ are provided.}} Given the latest approximation to the solution $Y_0\approx u^{k-1}(0, X_0), Z_t \approx v^{k-1}(t, X_t)$ and approximation to the distribution dependence $\widehat m_1^{k-1}, \widehat m_2^{k-1}, \widehat m_3^{k-1}$ (see the second step for how these functions are constructed; when $k=1$, these functions are defined by the neural networks (NNs) with randomly initialized weights), we consider the following forward SDEs for $(\tilde X_t^{k-1}, \tilde Y_t^{k-1})$ (the superscript $k-1$ denotes the $(k-1)^{th}$ stage):
\begin{equation}
\small{\begin{dcases}\label{eq:fsde_simulation}
      \mathrm{d} \tilde X_t^{k-1} = b(t, \tilde \Theta_t^{k-1}, \widehat{m}_1^{k-1}(t, \tilde X^{k-1}_t)) \ud t + \sigma(t, \tilde X_t^{k-1}) \ud W_t, \; & \tilde X_0^{k-1} = \xi,\\
    \mathrm{d} \tilde Y_t^{k-1} = - h(t, \tilde \Theta_t^{k-1}, \widehat{m}_2^{k-1}(t, \tilde X^{k-1}_t)) \ud t + v^{k-1}(t, \tilde X^{k-1}_t) \ud W_t,\; & \tilde Y_0^{k-1}= u^{k-1}(0, \tilde X_0^{k-1}).
  \end{dcases}}
\end{equation}
and use the corresponding distribution
\begin{equation*}
    \nu_t^{k-1} = \mc{L}(\tilde \Theta_t^{k-1}), \quad \mu_T^{k-1} = \mc{L}(\tilde X_T^{k-1}),
\end{equation*}
for approximating $\mc{L}(\Theta_t)$ and $\mc{L}(X_T)$. In practice, we will further approximate $\nu_t^{k-1}$ and  $\mu_T^{k-1}$ by their empirical versions, as detailed in Section~\ref{sec:alg_details}.

\paragraph{Step~2. Approximate the distribution dependence $m_1$, $m_2$ and $m_3$} 
Given the latest estimates of $\mc{L}(\Theta_t)$ and $\mc{L}(X_T)$ , we can view $m_1, m_2$ as functions of $(t, x)$ and $m_3$ as a function of $x$. Naturally, we can parameterize these functions using neural networks and optimize through supervised learning. After the $(k-1)^{th}$ stage, we optimize
\begin{equation}\label{def:learn_m_nn}
    \begin{aligned}
     &\inf_{\mathfrak{m}_1} \int_0^T\EE\|m_1(t, \tilde X_t^{k-1}, \nu_t^{k-1}) - \mathfrak{m}_1(t, \tilde X_t^{k-1})\|^2 \ud t, \\
    &\inf_{\mathfrak{m}_2} \int_0^T\EE\|m_2(t, \tilde X_t^{k-1}, \nu_t^{k-1}) - \mathfrak{m}_2(t, \tilde X_t^{k-1})\|^2 \ud t, \\
    & \inf_{\mathfrak{m}_3}  \EE\|m_3(\tilde X_T^{k-1}, \mu_T^{k-1}) - \mathfrak{m}_3(\tilde X_T^{k-1})\|^2,
    \end{aligned}
\end{equation}
where $\mathfrak{m_i}$, $i = 1, 2, 3$ are searched over a class of NNs (to be specified in Section~\ref{sec:numerics}). The optimized NNs, denoted by $\widehat m_1^{k}$, $\widehat m_2^{k}$ and $\widehat m_3^{k}$ respectively,  will be needed in Step 3 below, and Step 1 in the next stage.

\paragraph{Step~3. Solve the standard FBSDEs}
The FBSDEs to be solved at the $k^{th}$ stage for $(X_t^k, Y_t^k, Z_t^k)$ read as 
\begin{equation}\label{def:algo-fbsde}
    \begin{dcases}
     \ud X_t^k = b(t, \Theta_t^k, \widehat m_1^k(t, X_t^{k}))\ud t + \sigma(t, X_t^k)\ud W_t, & X^k_0 = \xi,\\
    \ud Y^k_t = - h(t, \Theta_t^k, \widehat m_2^k(t, X_t^{k}))\ud t + Z^k_t\ud W_t, & Y^k_T= g(X^k_T, \widehat m_3^k(X^{k}_T)),
    \end{dcases}
\end{equation}
and one aims to find the optimal NNs $\psi^k$ and $\phi^k$ that parameterize $Y_0^k$ and $Z^k_t$, respectively. The exact way of obtaining $\psi^k$ and $\phi^k$ depends on the particular BSDE solver that one will use, and this will be presented in Section~\ref{sec:BSDEsolver}. The functions $\psi^k$ and $\phi^k$ are considered as approximations to the decoupling field $u^k$ and $v^k$ of equation~\eqref{def:MV-FBSDE-fp}.

\begin{rem}
    A novel step in our algorithm is Step 2 (the supervised learning on $m_i$). It produces $m_i$'s proxy $\widehat m_i$, $i = 1, 2, 3$, thus provides quick evaluations of $m_i$. This is crucial in our ``Step 3: Solve the Standard FBSDEs''. If such $\widehat m_i$ is not available, one needs to solve the following system 
\begin{equation*}
    \begin{dcases}
     \ud X_t^k = b(t, \Theta_t^k,  m_1(t, X_t^{k}, \mc{L}(\Theta_t^k)))\ud t + \sigma(t, X_t^k)\ud W_t, & \hspace{-8pt} X^k_0 = \xi,\\
    \ud Y^k_t = - h(t, \Theta_t^k,  m_2(t, X_t^{k}, \mc{L}(\Theta_t^k))\ud t + Z^k_t\ud W_t, & \hspace{-8pt} Y^k_T= g(X^k_T,  m_3(X^{k}_T, \mc{L}(X_T^k))),
    \end{dcases}
\end{equation*}
in Step 3. However, by its definition (the dependence on $\mc{L}(\cdot)$), it can only be approximated via empirical distributions, thus very time-consuming compared to the quick evaluation of $\widehat m_i(t, x)$.
\end{rem}

\begin{rem}
Notice that when we learn $(\widehat{m}_1^k, \widehat{m}_2^k, \widehat{m}_3^k)$ in the optimization problem \eqref{def:learn_m_nn}, the difference is evaluated along $\tilde X_t^{k-1}$; while when we solve the FBSDEs \eqref{def:algo-fbsde}, we plug in $X_t^k$. The mismatch between the distribution of the random variables that we use to learn the functions and the random variables that we use to solve FBSDEs will lead to theoretical and numerical difficulty, and this is the very reason why we limit $(b,\sigma,h,g)$ in the forms \eqref{assump:algo2}. Under this assumption, the distribution of $X_t^k$ is absolutely continuous with respect to the distribution of $\tilde{X}_t^{k-1}$ through the Girsanov theorem; hence we can use the errors under $\tilde{X}_t^{k-1}$ to control the errors under $X_t^k$; see Lemma~\ref{lem: Girsanov} for a detailed statement.

An alternative approach is that one could use $(\widehat{m}_1^k, \widehat{m}_2^k, \widehat{m}_3^k)$ in the other two steps the same way as in the second step,
{\it i.e.}, replace
$(\Theta_t^{k-1}, X_t^{k-1})$ by $(\Theta_t^{k-2}, X_t^{k-2})$ in equation~\eqref{eq:fsde_simulation} and $(\Theta_t^{k}, X_t^{k})$ by $(\Theta_t^{k-1}, X_t^{k-1})$ in equation~\eqref{def:algo-fbsde}
in model coefficients $(b, \sigma, h, g)$. In this case, we will use the same random variables to learn the functions and to solve the FBSDEs and hence can handle more general setting on $(b,\sigma,h,g)$:
\begin{equation*}
    (b, h, \sigma) = (b, h, \sigma)(t, \Theta_t, m_1(t, \Theta_t, \MCL(\Theta_t))), \quad 
     g = g(X_T, m_2(X_T, \MCL(X_T))).
\end{equation*}
However, if one wants to simulate a path of $X_t^k$, one needs to first simulate paths of $X_t^0, X_t^1, \dots, X_t^{k-1}$ under the same realization of Brownian motions, which will not be computationally efficient.

In a nutshell, there is a trade-off between the two choices, and we decide to go with the current choice \eqref{assump:algo2} for the analysis in Section~\ref{sec:theory}, since it is more numerically favorable.

\end{rem}

\subsection{BSDE Solvers}\label{sec:BSDEsolver}
We now describe the numerical algorithm for solving the FBSDEs~\eqref{def:algo-fbsde} in the inner stage of fictitious play.
For the clarity of notations, we use generic FBSDEs
\begin{equation}\label{def:FBSDE}
\begin{dcases}
  \ud X_t = B(t, X_t, Y_t, Z_t) \ud t + \Sigma(t, X_t, Y_t, Z_t) \ud W_t, \quad &X_0 = x_0,  \\
  \ud Y_t = -H(t, X_t, Y_t, Z_t) \ud t + Z_t \ud W_t, \quad &Y_T = G(X_T),
\end{dcases}
\end{equation}
to describe the algorithm. We shall mainly consider two types of BSDE solvers that leverage neural networks and thus are suitable for high-dimensional settings.

The first algorithm is the Deep BSDE method \cite{weinan2017deep,han2018solving,han2020convergence1}, the first method designed for solving high-dimensional BSDEs and parabolic PDEs with neural networks (NNs).
The Deep BSDE method solves the variational problem which is equivalent to \eqref{def:FBSDE},
\begin{align}\label{eq:FBSDE_variation}
    &\inf_{Y_0, (Z_t)_{t \in [0,T]}} \EE\|G(X_T) - Y_T\|^2, \\
    s.t. \quad & X_t = x_0 + \int_0^t B(s, X_s, Y_s, Z_s)\ud s + \int_0^t \Sigma(s, X_s, Y_s, Z_s) \ud W_s, \notag \\
    \quad & Y_t = Y_0 - \int_0^t H(s, X_s, Y_s, Z_s) \ud s + \int_0^t Z_s \ud W_s. \notag
\end{align}
To solve the above variational problem numerically, the method works on a discretized version and parameterizes $Y$ and $Z$ at each $t_i$ as functions of $X_{t_i}$ using NNs.
To fix the notation, we consider a partition $\pi$ on the interval $[0,T]$ of size $N_T$: $0 = t_0 < t_1 < \cdots < t_i < \cdots < t_{N_T} = T$ and define $\Delta t_i := t_{i+1}- t_i$. Then we only need to minimize the $L^2$ difference between the simulated $Y_T$ and its target value $G(X_T)$:  
\begin{align}
  &\inf_{\psi \in \mc{N}', \{\phi_i \in \mc{N}_i\}_{i=0}^{N_T-1}} \EE\|G(X_T) - Y_T\|^2, \label{eq:deepBSDE-goal} \\
  s.t. &\quad X_0 = x_0, \quad Y_0 = \psi(X_0), \quad Z_{t_i} = \phi_i(X_{t_i}, Y_{t_i}), \label{eq:deepBSDE-Zt}\\
  & \quad X_{t_{i+1}} = X_{t_i} + B(t_i, X_{t_i}, Y_{t_i}, Z_{t_i}) \Delta t_i+ \Sigma(t_i, X_{t_i}, Y_{t_i}, Z_{t_i}) \Delta W_{t_i}, \label{eq:deepBSDE-Xt}\\
  & \quad Y_{t_{i+1}} = Y_{t_i} - H(t_i, X_{t_i}, Y_{t_i}, Z_{t_i}) \Delta t_i + Z_{t_i} \Delta W_{t_i}, \label{eq:deepBSDE-Yt}
\end{align}
where $\Delta W_{t_i} = W_{t_{i+1}}- W_{t_i}$ is the Brownian increment from $t_i$ and $t_{i+1}$, an $\psi$ and $\phi_i$ are functions in the hypothesis spaces $\mc{N}'$ and $\mc{N}_i$ that are generated by neural networks (NNs). Note that the above choice without $\phi_i$'s dependence on $Y$ is consistent with our goal of searching for the decoupling field (cf. \eqref{def:decouplingfield}). However, in the rigorous justification of Deep BSDE \cite{han2020convergence1}, $Y$ is needed for $\phi_i$ in the proof of coupled FBSDEs, thus we include it here.

The second algorithm we consider is the deep backward dynamic programming (DBDP) method \cite{hure2020deep}. The original DBDP \cite{hure2020deep} can only deal with the decoupled FBSDE in~\eqref{def:FBSDE}, that is, $B$ and $\Sigma$ are free of $(Y_t, Z_t)$, so does the presentation below. Accordingly, so far we can only apply the DBDP to MV-FBSDE problems in which $b$ and $\sigma$ are free of $(Y_t, Z_t)$ \footnote{The dependence on $\mc{L}(Y_t, Z_t)$ is fine as 
it will be approximated by $\widehat m_1^k$ that does not depend on $\mc{L}(Y_t, Z_t)$. This is exactly the setting of the example in Section~\ref{sec:numerics-benchmark}}. Same as the Deep BSDE method, the DBDP method parameterizes $Y_{t_i}$ and $Z_{t_i}$ at each $t_i$ as functions of $X_{t_i}$ and solves optimization problems to get the optimal parameters. Instead of solving the global minimization problem \eqref{eq:deepBSDE-goal} over the whole time interval in the Deep BSDE method, the DBDP method solves $Y_{t_i}$ and $Z_{t_i}$ sequentially and backwardly for $i = t_{N_T}-1, \ldots, 0$. Given the learnt functions $\psi^\ast_{t_{i+1}}$ representing $Y_{t_i+1}$ (by definition $Y_T$ is replaced by $g(X_t)$ at $i=t_{N_T}$), the $i^{th}$ problem is to train the neural networks $\psi_i, \phi_i$ via:
\begin{align}
	&	\inf_{\psi_i \in \mc{N}'_i, \phi_i \in \mc{N}_i} \EE \Big\|Y_{t_{i+1}} - Y_{t_i}  + H(t_i, {X}_{t_i}, Y_{t_i}, Z_{t_i}) \Delta t_i  - Z_{t_i} \Delta {W}_{t_{i}}\Big\|^2 \label{eq:DBDP-goal}\\
	s.t. & \quad Y_{t_i} = \psi_i(X_{t_i}), \quad Y_{t_{i+1}}^{\ast} = \psi_{i+1}^{ \ast}(X_{t_{i+1}}), \quad Z_{t_i} = \phi_i(X_{t_i}), \\
	& \quad X_{t_{i+1}} =  X_{t_i} + B(t_i, X_{t_i}) \Delta t_i + \Sigma(t_i, X_{t_i}) \Delta W_{t_i}. \label{eq:DBDP-Xt}
	\end{align}

In practice, the objective \eqref{eq:deepBSDE-goal} or \eqref{eq:DBDP-goal} in these two methods is approximated by its Monte-Carlo counterpart and the (near-)~optimal NN parameters that represent $\psi^\ast$ and $\phi_i^\ast$ (or $\psi_i^{\ast}$ and $\phi_i^{\ast}$ in DBDP) are obtained by stochastic gradient decent algorithms which require simulating a batch of Brownian paths at each update. To related to notations in Section~\ref{sec:algorithm_outline} Step 3, they provide $\phi^k$ and $\psi^k$. 

We remark that in the setting of MV-FBSDEs, the functions $(b, h, g)$ may depend on $\MCL(\Theta_t)$ in complicated forms, which may bring challenges to solve the optimization problems \eqref{eq:deepBSDE-goal} or \eqref{eq:DBDP-goal}. That is the reason we propose to use neural networks to approximate the dependence on the laws through \eqref{def:learn_m_nn} to make the resulting FBSDE easier to solve. For further numerical details and theoretical foundations of these two methods, we refer the readers to \cite{weinan2017deep,han2018solving,han2020convergence1,hure2020deep}.

\subsection{Implementation Details}
\label{sec:alg_details}
We now detail the algorithm outlined in Section~\ref{sec:algorithm_outline}. How the standard FBSDEs in the third step are numerically solved has been discussed in Section~\ref{sec:BSDEsolver}. So we will mainly present the numerical details of the other two steps. Same to the notations used in Section~\ref{sec:BSDEsolver}, we consider a partition $\pi$ on the interval $[0,T]$ of size $N_T$: $0 = t_0 < t_1 < \cdots < t_i < \cdots < t_{N_T} = T$ with $\Delta t_i := t_{i+1}- t_i$.
To update the distribution in the first step, we generate $N$ samples of initial condition $\xi^{n,{k-1}}$ and $N$ samples of Brownian motions $W_t^{n, {k-1}}$, $n = 1, \ldots, N$ to simulate the forward SDEs~\eqref{eq:fsde_simulation} in discrete time:
\begin{equation}\label{eq:fsde_simulation_discrete}
 \hspace{-5pt} \begin{dcases}
      \tilde X_{t_{i+1}}^{n, k-1} = \tilde X_{t_{i}}^{n, k-1} + b(t_i, \tilde \Theta_{t_i}^{n,k-1}, \widehat{m}_1^{k-1}(t_i, \tilde X_{t_i}^{n, k-1})) \Delta t_i + \sigma(t_i, \tilde X_{t_i}^{n, k-1}) \Delta W_{t_i}^{n,k-1}, \\
      \tilde Y_{t_{i+1}}^{n, k-1} = \tilde Y_{t_{i}}^{n, k-1} - h(t_i, \tilde \Theta_{t_i}^{n, k-1}, \widehat{m}_2^{k-1}(t_i, \tilde X_{t_i}^{n, k-1})) \Delta t_i + v^{k-1}(t_i, \tilde X_{t_i}^{n, k-1}) \Delta W_{t_i}^{n,k-1},
  \end{dcases}
\end{equation}
with initial conditions $\tilde{X}_0^{n, k-1} = \xi^{n,k-1}, \tilde{Y}_0^{n,k-1}= u^{k-1}(0, \tilde X_0^{n,k-1})$ and define the empirical measures
\begin{equation}\label{def:empirical}
    \nu_{t_i}^{k-1} = \frac{1}{N} \sum_{n=1}^N \delta_{\tilde \Theta_{t_i}^{n, k-1}}, \quad \mu_T^{k-1} = \frac{1}{N} \sum_{n=1}^N \delta_{\tilde X_T^{n, k-1}},
\end{equation}
for approximating $\mc{L}(\Theta_t)$ at discretized timestamps $t_i$ and $\mc{L}(X_T)$.

Next, we will use data obtained from~\eqref{eq:fsde_simulation_discrete} to construct the supervised learning problem in the second step to approximate the distribution dependence. By Monte Carlo sampling, we can approximate the problems in \eqref{def:learn_m_nn} with
\begin{equation}\label{def:learn_m_nn_discrete}
\begin{aligned}
    &\inf_{\mathfrak{m}_1} \sum_{1\leq n \leq N}\sum_{0\leq i \leq N_T-1}\|m_1(t_i, \tilde X_{t_i}^{n,k-1}, \nu_{t_i}^{k-1}) - \mathfrak{m}_1(t_i, \tilde X_{t_i}^{n,k-1})\|^2, \\
    &\inf_{\mathfrak{m}_2} \sum_{1\leq n \leq N}\sum_{0\leq i \leq N_T-1} \|m_2(t_i, \tilde X_{t_i}^{n, k-1}, \nu_{t_i}^{k-1}) - \mathfrak{m}_2(t_i, \tilde X_{t_i}^{n,k-1})\|^2, \\
    & \inf_{\mathfrak{m}_3}  \sum_{1\leq n \leq N}\|m_3(\tilde X_T^{n,k-1}, \mu_T^{k-1}) - \mathfrak{m}_3(\tilde X_T^{n,k-1})\|^2.
\end{aligned}
\end{equation}
Given $(\widehat{m}_1^{k},\widehat{m}_2^{k},\widehat{m}_3^{k})$ solved from \eqref{def:learn_m_nn_discrete}, we can solve the FBSDEs~\eqref{def:algo-fbsde} using the introduced BSDE solvers. Taking the Deep BSDE method as the example, the whole deep learning algorithm for solving the MV-FBSDEs can be summarized below.

\begin{algorithm}
    \caption{The Deep MV-FBSDEs method}
  \begin{algorithmic}[1]
  \Require{Input:} the time partition $0 = t_0 < t_1 < \cdots < t_i < \cdots < t_{N_T} = T$, the number of paths $N$ for simulating the SDEs, the initialization of the decoupling field $u^{0}(0, X_0), v^0(t, X_t)$ and the distribution dependence $(\widehat{m}_1^0,\widehat{m}_2^0,\widehat{m}_3^0)$
     \For{$k= 1, 2, \dots$}
        \State Simulate the forward SDEs~\eqref{eq:fsde_simulation_discrete} to collect the data $\tilde X_{t_i}^{n,k-1}, \tilde Y_{t_i}^{n,k-1}, \tilde Z_{t_i}^{n,k-1}$, $n=1,\dots, N, i=0, \dots, N_T$
        \State Define the empirical measures according to \eqref{def:empirical}
        \State Train $(\widehat{m}_1^{k},\widehat{m}_2^{k},\widehat{m}_3^{k})$ to solve the optimization problem~\eqref{def:learn_m_nn_discrete}
        \State Train $\psi^{k}$ and $\phi^{k}$ to solve the variational problem~\eqref{eq:deepBSDE-goal} corresponding to the FBSDEs~\eqref{def:algo-fbsde}
     \EndFor
  \end{algorithmic} 
  \label{alg:game}
\end{algorithm}

\section{Convergence Analysis}\label{sec:theory}
This section is dedicated to the numerical analysis of the algorithms proposed in Section~\ref{sec:algorithm}, in particular, the justification of learning $\widehat m_i^k$, $i = 1, 2, 3$ from \eqref{def:learn_m_nn_discrete}.
The metric we use here for measuring the convergence of distributions in $\MCP^2(\RR^\theta)$  falls into the category of integral probability metrics \cite{muller1997integral}, under which we can show various convergence speeds being free of the dimensionality $\theta$. We give a brief review below for the readers' convenience.

\subsection{A Class of Integral Probability Metrics (IPMs)}\label{sec:GMMD}
Firstly introduced in \cite{han2021class}, the integral probability metric defines a metric on the space of probability measures $\MCP(E)$ based on a class of test functions.
\begin{defn}
Let $\Phi$ be a class of measurable functions on a Euclidean space $E$ such that 
\begin{equation*}
    \sup_{f \in \Phi} |f(x)| \leq C(1 + \|x\|),
\end{equation*}
for a constant $C >0$ depending on $\Phi$. Then for any $\mu$, $\mu' \in \MCP(E)$, the IPM $D_\Phi$ is defined by  
\begin{equation*}
    D_{\Phi}(\mu, \mu') = \sup_{f \in \Phi} | \int_{E} f \ud (\mu - \mu')|.
\end{equation*}
\end{defn}

\begin{assu}[Test function class]\label{function_class_assumption}
Assume ${\Phi}$ satisfies the following properties:
\begin{enumerate}
    \item[(a)] If $\mu$ is a signed measure on $E$, 
    \begin{equation*}
        \int_{E}f \rmd \mu = 0, \; \forall f \in {\Phi} \Rightarrow \mu \equiv 0; 
    \end{equation*}
    \item[(b)] There exists a constant $A > 0$, such that for any $f \in {\Phi}$, $\mathrm{Lip}(f) \le A$, $|f(0)|\le A$, and 
   for any $\mathcal{X}= \{x^1,\dots,x^N\} \subset E$, the empirical Rademacher complexity satisfies
    \begin{equation}
        \mathrm{Rad}_N({\Phi},\mathcal{X}) \coloneqq \frac{1}{N}\EE \sup_{f \in {\Phi}} |\sum_{i=1}^N\xi_i f(x^i)| \le \frac{A}{N} \sqrt{\sum_{i=1}^N(\|x^i\|^2 + 1)}, 
    \end{equation}
    where $\xi_1,\dots,\xi_N$ are i.i.d. random variables drawn from the Rademacher distribution, {\it i.e.}, $\mathbb{P}(\xi_i = 1) = \mathbb{P}(\xi_i = -1) = \frac{1}{2}$.
\end{enumerate}
\end{assu}

Then under Assumption~\ref{function_class_assumption}, it is shown (\emph{cf.} \cite{han2021class}) that the convergence rate of empirical measures associated from independent samples drawn from a given distribution $\mu \in \MCP^2(E)$, measured by $D_\Phi$, is free of the dimensionality of $E$.

\begin{thm}[{\cite[Theorem 2.1]{han2021class}}]\label{thm: GMMD_error}
Under Assumption~\ref{function_class_assumption}, $D_\Phi$ is a metric on $\mc{P}^1(E)$. For any $\mu_1$ and $\mu_2 \in \mathcal{P}^1(E)$,
\begin{equation}\label{Wasserstein_relation}
    D_\Phi(\mu_1,\mu_2) \le A\mathcal{W}_1(\mu_1,\mu_2),
\end{equation}
where $\mc{W}_1$ denotes the 1-Wasserstein metric. In addition,  Given $\mu \in \mathcal{P}^2(E)$, let $X^1,\dots,X^N$ be i.i.d. random variables drawn from the distribution $\mu$ and 
\begin{equation*}
    \bar{\mu}^N = \frac{1}{N}\sum_{i=1}^N \delta_{X^i}
\end{equation*}
be the empirical measure of $X^1, \ldots, X^N$. Then,
\begin{align*}
    &\EE \rmD_{{\Phi}}(\mu,\bar{\mu}^N)
    \le \frac{2A}{\sqrt{N}} \sqrt{\int_{\bR^d}[\|x\|^2+1] \rmd \mu(x)},\\
    &\EE \left[\rmD_{\Phi}^2(\mu, \bar{\mu}^N)\right]\le  \frac{5A^2}{N}\int_{\bR^d}[\|x\|^2+1] \rmd \mu(x).
\end{align*}

\end{thm}

A similar result holds for measures on the space of continuous functions $\hspace{-1pt}C([0,T]; \hspace{-2pt}E)$ with additional assumptions regarding functions' modulus of continuity. Test function classes particularly discussed therein include the reproducing kernel Hilbert space (RKHS), the Barron space, and flow-induced function spaces \cite{weinan2019barron}. Direct applications of metrics induced by such test function classes are that, the approximation error to the Nash equilibrium of finite player games by its mean-field limit is independent of the dimensionality of $E$, and that the empirical distribution associated with a finite-particle system to its McKean-Vlasov limit is free of CoD. While if measured by Wasserstein metric $\mc{W}_1$ or $\mc{W}_2$, these convergence rates decrease as the dimension of $E$ increases. 

\subsection{Numerical Analysis}
Our analysis will focus on the following Mckean-Vlasov FBSDEs for $(X_t,Y_t,Z_t) \in \mathbb{H}^2(\bR^d \times \bR^p\times \bR^{p\times q})$:
\begin{align}\label{def:MV-FBSDE-simple}
\begin{dcases}
  \ud X_t = b(t, X_t,m_1(t,X_t,\MCL(X_t))) \ud t + \sigma(t, X_t) \ud W_t, &X_0 = \xi,  \\
  \ud Y_t = -h(t, \Theta_t, m_2(t,X_t,\MCL(X_t, Y_t))) \ud t + Z_t \ud W_t, &Y_T = g(X_T, m_3(X_T,\MCL(X_T))).
\end{dcases}
\end{align}
For the coupled Mckean-Vlasov FBSDEs, that is, $b$ depends on $Y$, $Z$ or $\mc{L}(X,Y)$, similar results can be obtained under small duration conditions or the existence of global Lipschitz decoupling fields, like the situations in \cite{han2020convergence,han2020convergence1,reisinger2020path}. If $m_1$ and $m_2$ also depend on $\mc{L}(X,Y,Z)$, then we can also obtain the result following the method in \cite[Section~4.2]{han2020convergence} if we have the following regularity \cite{reisinger2020path} with respect to $Z$,
\begin{equation*}
    \EE\|Z_t - Z_{t'}\|_F^2 \le C\|t -t'\|,
\end{equation*}
where $\|\cdot\|_F$ is the Frobenius norm. 
 
\begin{assu}\label{assumption bsde}
\begin{enumerate}[(a)]
    \item The functions $b(t, x, m): [0, T]\times \bR^d \times \bR^l \to \bR^d $, $\sigma(t, x): [0,T] \times \bR^d  \to \bR^{d \times q}$, $h(t, x, y,z, m): [0,T] \times \bR^\theta \times \RR^l \to \bR^p$ and $g(x,m): \bR^d \times \RR^l \to \bR^p$ are Lipschitz with respect to $x,y,z$ and $m$, with a Lipschitz constant $L$:
    \begin{align*}
        \|b(t,x,m) - & b(t,x',m')\|^2 + \|\sigma(t,x) - \sigma(t,x')\|_F^2  \\
         &+ \|h(t,x, y, z, m) - h(t, x', y', z',m')\|^2 + \|g(x,m) - g(x,m')\|^2  \\
        & \qquad \le  L[\|x-x'\|^2 + \|y-y'\|^2 + \|z-z'\|_F^2 + \|m - m'\|^2].
    \end{align*}
   
    \item The functions $m_1(t,x,\mu):[0,T]\times\bR^d\times\MCP^2(\RR^d)\rightarrow \RR^l$, $m_2(t,x,\nu):[0,T]\times\RR^d\times\MCP^2(\RR^d\times\RR^p)\rightarrow \RR^l$ and $m_3(x,\mu):\bR^d \times\MCP^2(\RR^d)\rightarrow \RR^l$ are Lipschitz with respect to $x$, $\mu$ and $\nu$, with the same constant $L$:
    \begin{multline*}
       \|m_1(t,x,\mu) - m_1(t,x',\mu')\|^2 + \|m_2(t,x,\nu)  - m_2(t,x',\nu')\|^2 \\
       + \|m_3(x,\mu) - m_3(x',\mu')\|^2 
       \le L[\|x - x'\|^2 + \rmD_{\Phi}^2(\mu,\mu') +\rmD_{\Phi'}^2(\nu,\nu')],
    \end{multline*}
    where $\Phi \subset \RR^d$ and $\Phi' \subset \RR^d \times\RR^p$ are two classes of functions for the IPM and satisfy Assumption~\ref{function_class_assumption}. 
    \item The functions $b$, $\sigma$, $h$, $g$, $m_1$ and $m_2$ is 1/2-H\"older continuous with respect to $t$. We will use $K$ to denote the H\"older constant.
    \item There exists a constant $K$, such that
    \begin{multline*}
        \|b(t,0, 0)\|^2 + \|\sigma(t,0)\|_F^2 +  \|h(t,0,0)\|^2 +  \|g(0,0)\|^2 +  \|m_1(t,0,\delta_0)\|^2 \\
        +\|m_2(t,0,\delta_0)\|^2 +\|m_3(0,\delta_0)\|^2+ \EE\|\xi\|^2 \le K,
    \end{multline*}
    where $\delta_0$ denotes the Dirac measure at $0$.
\end{enumerate}
\end{assu}
\begin{lem}\label{thm: wellpose}
Under Assumption~\ref{assumption bsde}, the MV-FBSDEs \eqref{def:MV-FBSDE-simple} possess a unique adapted solution (i.e., $(X_t, Y_t, Z_t)\in \mb{S}^2(\RR^d)\times \mb{S}^2(\RR^p) \times \mb{H}^2(\RR^{p\times q})$ are adapted to the filtration $\mb{F}= (\mc{F}_t)_{t \in [0,T]}$) such that
\begin{equation*}
    \sup_{0\le t \le T}[\EE\|X_t\|^2 + \EE\|Y_t\|^2] + \int_{0}^T\EE\|Z_t\|_F^2\rmd t \le C,
\end{equation*}
where $C > 0$ is a constant depending only on $A$, $L$, $K$ and $T$. 
\end{lem}
 \begin{proof}
 Theorem \ref{thm: GMMD_error} and Assumption \ref{assumption bsde} together imply that the MV-FBSDEs \eqref{def:MV-FBSDE-simple} satisfy the standard assumption in \cite[Sections~4.2.1 and 4.2.2]{CaDe1:17}. Therefore, with \cite[Theorems 4.21 and 4.23]{CaDe1:17}, we obtain the desired result.
\end{proof}

Since $\mc{L}(X_t, Y_t, Z_t)$ are deterministic functions of time $t$, we can view $m_1$ and $m_2$ in \eqref{def:MV-FBSDE-simple} as functions of $(t, x)$, and similarly view $m_3$ as a function of $x$. We define
\begin{equation}\label{def:mtilde}
    \tilde m_1(t, x) = m_1(t, x, \mc{L}(X_t)), \; \tilde m_2(t, x) = m_2(t, x, \mc{L}(X_t, Y_t)), \; \tilde m_3(x) = m_3(x, \mc{L}(X_T)),
\end{equation}
where $(X_t, Y_t)$ is the solution to \eqref{def:MV-FBSDE-simple}.
This motivates us, in practice, to use neural networks to parameterize $m_1, m_2$ and $m_3$ with inputs of $(t, x)$ and $x$, as stated in \eqref{def:learn_m_nn}. To analyze the error associated with fictitious play, we define the following function sets $\mc{M}_1$ and $\mc{M}_2$:  given a fixed constant $M > 0$,
\begin{align*}
    &\mathcal{M}_1 = \{m:[0,T]\times \RR^d\rightarrow \RR^l, \|m(t,x) - m(t',x')\|^2 \le M [|t-t'|+\|x - x'\|^2], \\
    & \hspace{200pt} \|m(t,x)\| \le M[1 + \|x\|]  \},\\
    &\mathcal{M}_2 =\{m:\RR^d \rightarrow \RR^l, \|m(x) - m(x')\|^2 \le M\|x - x'\|^2, \|m(x)\| \le M[1+\|x\|]\}.
\end{align*}
Let $\mathcal{M} = \mathcal{M}_1 \times \mathcal{M}_1 \times \mathcal{M}_2$.
We then define a map $\mathcal{T}: \mathcal{M} \rightarrow C([0,T]\times \RR^d) \times C([0,T]\times \RR^d)\times C(\RR^d)$ such that, for any $(\bar m_1, \bar m_2, \bar m_3) \in \mc{M}$:
\begin{equation}\label{def:T}
 \mathcal{T}(\bar{m}_1,\bar{m}_2,\bar{m}_3)(t,x) := (m_1(t,x,\MCL(\bar{X}_t)),m_2(t,x,\MCL(\bar{X}_t,\bar{Y}_t)),m_3(x,\MCL(\bar{X}_T))),
 \end{equation}
where $(\bar{X}_t,\bar{Y}_t,\bar{Z}_t)$ is the solution to the decoupled FBSDEs:
\begin{equation}\label{def:FBSDE-m}
    \begin{dcases}
  \ud \bar{X}_t = b(t, \bar{X}_t,\bar{m}_1(t,\bar{X}_t)) \ud t + \sigma(t, \bar{X}_t) \ud W_t, \quad \bar{X}_0 = \xi,  \\
  \ud \bar{Y}_t = -h(t, \bar{\Theta}_t, \bar{m}_2(t,\bar{X}_t)) \ud t + \bar{Z}_t \ud W_t, \quad \bar{Y}_T = g(\bar{X}_T, m_3(\bar{X}_T)).
\end{dcases}
\end{equation}
That is, $\mc{T}$ maps any element in $\mc{M}$ to the functions $(m_1(t, x, \cdot), m_2(t, x, \cdot), m_3(x, \cdot))$ with $\cdot$ being the law of the FBSDE system \eqref{def:FBSDE-m}. This definition corresponds to the idea of fictitious play that sequentially updates the quantity $Q^{(k)}$ of interest. 

The next lemma shows that if $M$ is large enough, then $\mathcal{T}$ is closed in $\mathcal{M}$.
\begin{lem}\label{thm_maps}
There exists a constant $M_0 > 0$, depending only on $A, L, K$ and $T$, such that for any $M \ge M_0$,
$\mathcal{T} m \in \mathcal{M}$ for any $m \in \mathcal{M}$.
\end{lem}
\begin{proof}

Notice that as long as $M \ge L$, the Lipschitz continuity with respect to $x$ is always kept.
By the definition of $\mathcal{T}$,  Assumption \ref{assumption bsde}(c) and inequality \eqref{Wasserstein_relation}, in order to show the H\"older continuity  of $(\mathcal{T}m_1,\mathcal{T}m_2)$ with respect to $t$, it suffices to prove that
\begin{equation*}
    \EE\|\bar{X}_t - \bar{X}_{t'}\|^2 + \EE\|\bar{Y}_t - \bar{Y}_{t'}\|^2 \le C\|t -t'\|,
\end{equation*}
which can be obtained by \cite[Theorem 5.2.2]{zhang2017backward}.
Here $C$ is a positive constant only depending on $A,L,K,T$.

 The linear growth property with respect to $x$ can be obtained using a similar method as in \cite[Section~4]{bender2008time} and we hence omit the proof.
\end{proof}

The next result focuses on how error changes in terms of the stage of fictitious play. Suppose we have a sequence of functions $(m_1^k, m_2^k, m_3^k) \in  \mathcal{M}$, $k \in \mathbb{N}$.
For each $k$,
let $(X_t^k, Y_t^k, Z_t^k) \in \mathbb{H}^2(\bR^d \times \bR^p\times \bR^{p\times q})$ solves
\begin{equation}\label{bsde:iteration}
    \begin{dcases}
  \ud X_t^k = b(t, X_t^k,m_1^k(t,X_t^k)) \ud t + \sigma(t, X_t^k) \ud W_t, \quad &X_0^k = \xi,  \\
  \ud Y_t^k = -h(t, \Theta_t^k, m_2^k(t,X_t^k)) \ud t + Z_t^k \ud W_t, \quad &Y_T^k = g(X_T^k, m_3^k(X_T^k)).
\end{dcases}
\end{equation}
The theorem below shows how $m^{k+1}_i$ and $m^k_i$, $i = 1, 2, 3$, should be related in order to achieve small errors. More discussions will be given after the proof.

\begin{thm}\label{thm_iteration}
Under Assumptions~\ref{function_class_assumption} and \ref{assumption bsde}, there exist constants $C > 0$ and $0 < q < 1$ which only depend on $A, L, K, T$ and $M$, such that
\begin{multline}\label{thm3:main_result}
    \sup_{0 \le t \le T}\big[\EE\|X_t - X_t^k\|^2 + \EE\|Y_t - Y_t^k\|^2\big]  + \int_{0}^T\EE\|Z_t - Z_t^k\|_F^2 \rmd t\\
    \le C\Big\{q^k +\sum_{j=0}^{k-1}q^{k-j}\Big[\int_{0}^T\EE\|(m_1^{j+1},m_2^{j+1})(t,X_t^{j+1}) - (\tilde{m}_1^{j},\tilde{m}_2^{j})(t,X_t^{j+1})\|^2\rmd t \\
     + \EE\|m_3^{j+1}(X_T^{j+1}) - \tilde{m}_3^j(X_T^{j+1})\|^2\Big] \Big\},
\end{multline}
where $(X_t, Y_t, Z_t)$ solves the MV-FBSDEs \eqref{def:MV-FBSDE-simple}, $(X_t^k, Y_t^k, Z_t^k)$ solves \eqref{bsde:iteration}, and \\$(\tilde{m}_1^k,\tilde{m}_2^k,\tilde{m}_3^k) \coloneqq \mathcal{T}(m_1^k,m_2^k,m_3^k)$.
\end{thm}
\begin{proof}
In the sequel, we will use $C$ to denote a positive constant depending only on $A, L, K, T$ and $M$, which may vary from line to line.

We then estimate the differences between $\Theta$ and $\Theta^k$. Recall $\tilde m_1, \tilde m_2, \tilde m_3$ from \eqref{def:mtilde}, for any given $k \ge 1$, we define $ \delta X_t^k = X_t - X_t^k$, $\delta Y_t^k = Y_t - Y_t^k$, $\delta Z_t^k = Z_t - Z_t^k$, and
\begin{align*}
    & I_1^k = \int_{0}^T\EE\|m_1^{k}(t,X_t^k) - \tilde{m}_1^{k-1}(t,X_t^k)\|^2 \rmd t,\\
    & I_2^k = \int_{0}^T\EE\|m_2^{k}(t,X_t^k) - \tilde{m}_2^{k-1}(t,X_t^k)\|^2 \rmd t,\\
    &I_3^k  = \EE\|m_3^{k}(X_T^k) - \tilde{m}_3^{k-1}(X_T^k)\|^2,\\
    &\delta h_t = h(t,\Theta_t,\tilde{m}_2(t,X_t)) - h(t,\Theta_t^k,m_2^k(t,X_t^k)). 
\end{align*}
Combining \eqref{def:MV-FBSDE-simple}, \eqref{bsde:iteration} and  \cite[ Theorem~3.2.4]{zhang2017backward}, we know that for any $t \in [0,T]$:
\begin{align}\label{thm3: eq1}
    \sup_{0\le s \le t}\EE\|\delta X_{s}^k\|^2 &\le C\int_{0}^t\EE\|b(s,X_{s}^k,\tilde{m}_1(s,X_{s}^k)) - b(s,X_{s}^k,m_1^k(s,X_{s}^k))\|^2\rmd s\\
    &\le C\int_{0}^t\EE\|\tilde{m}_1(s,X_{s}^k)-m_1^k(s,X_{s}^k)\|^2 \rmd s \notag\\
    &\le C\int_{0}^t\EE\|\tilde{m}_1(s,X_{s}^k) - \tilde{m}_1^{k-1}(s,X_{s}^k)\|^2\rmd s + CI_1^k  \notag\\
    &=C\int_{0}^t\EE\|m_1(s,X_{s}^k,\MCL(X_{s})) -m_1(s,X_{s}^k,\MCL(X_{s}^{k-1}))\|^2 \rmd s+CI_1^k  \notag\\
    &\le C\int_{0}^t\EE\rmD_{\Phi}^2(\MCL(X_{s}),\MCL(X_{s}^{k-1}))\rmd s +CI_1^k  \notag\\
    &\le C \int_{0}^t\EE\|\delta X_{s}^{k-1}\|^2 \rmd s + CI_1^k.  \notag
\end{align}
Then by induction, one has 
\begin{equation}\label{thm3: eq2}
    \int_{0}^{t}\EE\|\delta X_{s}^k\|^2\rmd s \le \frac{C^k}{k!}\int_{0}^t(t-s)^k\EE\|\delta X_{s}^0\|^2\rmd s + \sum_{j=1}^{k}\frac{(Cj)^k}{j!}I_1^{k+1-j}.
\end{equation}
Noticing that with $m_1^0 \in \mathcal{M}_1$, \cite[Theorem~3.2.2]{zhang2017backward} gives
\begin{equation*}
    \sup_{0\le t \le T}\EE\|X_t^0\|^2 \le C.
\end{equation*}
Combining it with Lemma~\ref{thm: wellpose} and the estimate \eqref{thm3: eq2} yields
\begin{equation*}
    \int_{0}^{T}\EE\|\delta X_t^k\|^2\rmd t \le \frac{(CT)^k}{k!} + \sum_{j=1}^k\frac{(CT)^j}{j!}I_1^{k+1-j}.
\end{equation*}
Noticing that there exists another constant $\tilde C > 0$ and $0 < q <1$ such that
\begin{equation*}
    \frac{(CT)^k}{k!} \le \tilde C q^k, \; \forall k \in \mathbb{N}^+,
\end{equation*}
therefore
\begin{equation*}
    \int_{0}^{T}\EE\|\delta X_t^k\|^2\rmd t \le Cq^k + C\sum_{j=0}^{k-1}q^{k-j}I_1^{j+1}.
\end{equation*}
Plugging the above inequality into \eqref{thm3: eq1} produces the $X$-part in \eqref{thm3:main_result}.
The analysis for $(Y,Z)$-part can be obtained in a similar way if one replaces \cite[Theorems~3.2.2 and 3.2.4]{zhang2017backward} in the above proof for $X$-part by \cite[Theorem~4.2.1 and 4.2.3]{zhang2017backward} for $(Y, Z)$-part.
\end{proof}

If we want to apply Theorem \ref{thm_iteration} to explain the convergence of Deep MV-FBSDE Algorithm, one issue is that in the right side of inequality \eqref{thm3:main_result} $X_t^{j+1}$ occurs while in the optimization problems \eqref{def:learn_m_nn} $X_t^j$ occurs. The following lemma shows the legitimacy of replacing $X_t^{j+1}$ by $X_t^j$.

\begin{lem}\label{lem: Girsanov}
Assume that 
\begin{equation*}
    \|\sigma(t,x)\|_S \le K, \quad \|\phi(t,x, m)\|_\infty \le K,
\end{equation*}
where $\|\cdot\|_S$ denotes the spectral norm on $\RR^{d\times q}$, $\|\cdot\|_\infty$ denotes the infinity norm on $\RR^q$, and $\phi$ satisfies that
\begin{equation*}
    b(t,x,m) = \sigma(t,x)\phi(t,x,m).
\end{equation*}

Let $\bar{m}^1$, $\bar{m}^2$ be two arbitrary functions in $\mathcal{M}_1$, and $\bar{X}^1$ and $\bar{X}^2$ be the associated solutions to the SDEs:
\begin{equation*}
    \rmd \bar{X}_t^i = b(t,\bar{X}_t^i,\bar{m}^i(t,\bar{X}_t^i))\rmd t + \sigma(t,\bar{X}_t^i)\rmd W_t, \quad \bar{X}_0^i = \xi, \quad i = 1, 2.
\end{equation*}
Let $m$, $m' \in \mathcal{M}_2$, then for any $\epsilon \in (0,1)$, there exists a constant $C(\epsilon)$, which only depends on $A,L,K,T,M$ and $\epsilon$ such that
\begin{equation*}
    \EE\|m(\bar{X}_t^1) - m'(\bar{X}_t^1)\|^2 \le C(\epsilon)\left[\EE\|m(\bar{X}_t^2) - m'(\bar{X}_t^2)\|^2\right]^\epsilon, \quad \forall t \in [0,T].
\end{equation*}
\end{lem}

\begin{proof}
The proof of this lemma is based on the Girsanov theorem; See, {\it e.g.}, \cite[Theorem~4]{han2020convergence}. 
\end{proof}

We are now ready to justify the Deep MV-FBSDE Algorithm. Here we take the version using the Deep BSDE solver as an instance. Similar results can also be established for the version using the DBDP solver.

We take a sequence $\{(\hat{m}_1^k,\hat{m}_2^k,\hat{m}_3^k)\}_{k \in \mathbb{N}} \in \mathcal{M}$, in which $M$ is large enough such that the closeness of $\mathcal{T}$ in Lemma \ref{thm_maps} holds. Note that here $\hat m_i^k$ are general functions, and can be, but not limited to, the functions $\widehat m_i^k$ defined in \textbf{Step 2} of Section~\ref{sec:algorithm_outline}. Recall the notations of partition $\pi$ on the interval $[0,T]$ of size $N_T$: $0 = t_0 < t_1 < \cdots < t_i < \cdots < t_{N_T}$ with $\Delta t_i = t_{i+1} - t_i$ and $\Delta W_{t_i} = W_{t_{i+1}} - W_{t_i}$.  Let $\{(u^k,v^k_i)\}_{k \in \mathbb{N}, 0 \leq i \leq N_T-1}$ be a sequence of functions such that  $u^k \in C(\bR^d;\bR^p)$ and $ v^k_i \in C(\bR^d;\bR^{p \times q})$ for any $k \in \mathbb{N}$ and $0 \le i \le N_T-1$, we define
\begin{equation}\label{eq_bsde_iteration}
\begin{cases}
    \tilde{X}_{t_{i+1}}^{k,\pi} = \tilde{X}_{t_i}^{k,\pi} +b(t_i,\tilde{X}_{t_i}^{k,\pi}, \hat{m}_1^k(t_i,\tilde{X}_{t_i}^{k,\pi}))\Delta t_i + \sigma(t_i,\tilde{X}_{t_i}^{k,\pi})\Delta W_{t_i},\\
    \tilde{Y}_{t_{i+1}}^{k,\pi} = \tilde{Y}_{t_i}^{k,\pi} -h(t_i,\tilde{\Theta}_{t_i}^{k,\pi}, \hat{m}_2^k(t_i,\tilde{X}_{t_i}^{k,\pi}))\Delta t_i + v^k_i(\tilde{X}_{t_i}^{k,\pi})\Delta W_{t_i},
\end{cases}
\end{equation}
with initial conditions $\tilde{X}_0^{k,\pi} = \xi$ and $\tilde{Y}_0^{k,\pi} = u^k(\tilde{X}_0^{k,\pi})$. The next theorem gives the convergence of the Deep MV-FBSDE Algorithm with the Deep BSDE solver.

\begin{thm}\label{main_thm}
Under Assumptions~\ref{function_class_assumption} and \ref{assumption bsde}, for any $\epsilon \in (0,1)$, there exist a constant $q \in (0,1)$ which only depend on $A, L, K, T$ and $M$ and $C(\epsilon) > 0$ which depend on $A, L,K,T,M$ and $\epsilon$, such that
\begin{align}
    \sup_{0 \le t \le T}&[\EE\|X_t - \tilde{X}_{\pi(t)}^{k,\pi}\|^2 + \EE\|Y_t - \tilde{Y}_{\pi(t)}^{k,\pi}\|^2] + \int_{0}^{T}\EE\|Z_t - v_i^k(\tilde{X}_{\pi(t)}^{k,\pi})\|_F^2\rmd t \notag\\
    &\le C(\epsilon)\Big\{q^k + \|\pi\| + \frac{1}{N}
    +  \EE\|\tilde{Y}_T^{k,\pi} - g(\tilde{X}_T^{k,\pi}, \hat{m}_3^k(\tilde{X}_T^{k,\pi}))\|^2 \notag\\
    & \quad  + \sum_{j=0}^{k-1}q^{k-j}\EE\|\tilde{Y}_T^{j,\pi} - g(\tilde{X}_T^{j,\pi}, \hat{m}_3^j(\tilde{X}_T^{j,\pi}))\|^2 \notag\\
 &   \quad + 
    \sum_{j=0}^{k-1}q^{k-j}\Big[ \EE\|\hat{m}_3^{j+1}(\tilde{X}_T^{j,\pi}) - m_3(\tilde{X}_T^{j,\pi},\mu_T^{j})\|^2 \notag\\
    & \quad +\sum_{i = 0}^{N_T - 1}q^{k-j}
    \EE\|\hat{m}^{j+1}_1(t_i,\tilde{X}_{t_i}^{j,\pi}) - m_1(t_i,\tilde{X}_{t_i}^{j,\pi}, \mu_{t_i}^j)\|^2\Delta t_i\Big], \notag\\
     & \quad +\sum_{i = 0}^{N_T - 1}q^{k-j}
    \EE\|\hat{m}^{j+1}_2(t_i,\tilde{X}_{t_i}^{j,\pi}) - m_2(t_i,\tilde{X}_{t_i}^{j,\pi}, \nu_{t_i}^j)\|^2\Delta t_i\Big] \Big\}^\epsilon,
   \label{thm_main:main}
\end{align}
in which $\|\pi\| := \max_{0 \le i \le N_T-1}(t_{i+1}-t_i)$, $\pi(t) = t_i$ if $t \in [t_i,t_{i+1})$ and $ \{\mu_{t_i}^k\}_{i=0}^{N_T}$ and $\{\nu_{t_i}^k\}_{i=0}^{N_T}$ are defined as the empirical distribution of $\tilde{X}_{t_i}^{k,\pi}$ and $(\tilde{X}_{t_i}^{k,\pi},\tilde{Y}_{t_i}^{k,\pi})$:
\begin{equation*}
    \mu_{t_i}^k = \frac{1}{N}\sum_{n=1}^{N}\delta_{\tilde{X}_{t_i}^{k,\pi,n}}, \;  \nu_{t_i}^k = \frac{1}{N}\sum_{n=1}^{N}\delta_{(\tilde{X}_{t_i}^{k,\pi,n},\tilde{Y}_{t_i}^{k,\pi,n})}, 
\end{equation*}
where 
\begin{align*}
  &\tilde{X}_{t_{i+1}}^{k,\pi,n} = \tilde{X}_{t_i}^{k,\pi,n} +b(t_i,\tilde{X}_{t_i}^{k,\pi,n}, \hat{m}_1^k(t_i,\tilde{X}_{t_i}^{k,\pi,n}))\Delta t_i + \sigma(t_i,\tilde{X}_{t_i}^{k,\pi,n})\Delta W_{t_i}^n, \; \tilde{X}_{0}^{k,\pi,n}  = \xi^n,  \\
  &\tilde{Y}_{t_{i+1}}^{k,\pi,n} = \tilde{Y}_{t_i}^{k,\pi,n}\hspace{-2pt} -\hspace{-2pt}h(t_i,\tilde{\Theta}_{t_i}^{k,\pi,n}, \hat{m}_2^k(t_i,\tilde{X}_{t_i}^{k,\pi,n}))\Delta t_i + v_i^k(\tilde{X}_{t_i}^{k,\pi,n})\Delta W_{t_i}^n, \tilde{Y}_{0}^{k,\pi,n}  = u^k(\xi^n),
\end{align*}
$\{W_t^n\}_{n=1}^N$ are i.i.d. Brownian motion in $\bR^q$ and $\{\xi^n\}_{n=1}^N$ are i.i.d. copies of $\xi$.

Let $(\tilde{X}_t^k,\tilde{Y}_t^k,\tilde{Z}_t^k)$ be the solution of 
\begin{equation}\label{def:FBSDE-mk}
    \begin{dcases}
  \ud \tilde{X}_t^k = b(t, \tilde{X}_t^k,\hat{m}_1^k(t,\tilde{X}_t^k)) \ud t + \sigma(t, \tilde{X}^k_t) \ud W_t, \quad \tilde{X}_0^k = \xi,  \\
  \ud \tilde{Y}^k_t = -h(t, \tilde{\Theta}_t^k, \hat{m}_2^k(t,\tilde{X}_t^k)) \ud t + \tilde{Z}^k_t \ud W_t, \quad \tilde{Y}^k_T = g(\tilde{X}^k_T, \hat{m}_3^k(\tilde{X}^k_T)),
\end{dcases}
\end{equation}
$\mathcal{N}'$ and $\{\mathcal{N}_i\}_{i=0}^{N_T-1}$ be any subsets of $C(\bR^d;\bR^p)$ and $C(\bR^d;\bR^{p\times q})$, respectively. Then, for any $k \in \mathbb{N}$,
\begin{multline}
     \inf_{u^k \in \mathcal{N}', v_i^k \in \mathcal{N}_i} \EE \|\tilde{Y}_T^{k,\pi} - g(\tilde{X}_T^{k,\pi},\hat{m}_3^k(\tilde{X}_T^{k,\pi}))\|^2 \\
    \le C\Big\{ \|\pi\|+ \inf_{u^k \in \mathcal{N}'}\EE\|\tilde{Y}_0^k - u^k(\tilde{X}_0^{k,\pi})\|^2 
    + \sum_{i=0}^{N_T - 1}\inf_{v_i^k \in \mathcal{N}_{i}}\EE\|\tilde{Z}_{t_i}^{k,\pi}- v_i^k(\tilde{X}_{t_i}^{k,\pi})\|^2\Big\},\label{bsde_error_upper}
\end{multline}
where $\tilde{Z}_{t_i}^{k,\pi} = (\Delta t_i)^{-1}\EE[\int_{t_i}^{t_{i+1}}\tilde{Z}_t^k\rmd t|\tilde{X}_{t_i}^{k,\pi}]$.

Finally, for any $k \in \mathbb{N}$,
\begin{align}
    \inf_{(\hat{m}_1^k,\hat{m}_2^k,\hat{m}_3^k) \in  \mathcal{M}}\Big\{&\sum_{i = 0}^{N_T - 1}
    \EE\|\hat{m}^{k+1}_1(t_i,\tilde{X}_{t_i}^{k,\pi}) - m_1(t_i,\tilde{X}_{t_i}^{k,\pi}, \mu_{t_i}^k)\|^2\Delta t_i \notag\\
    & + \sum_{i = 0}^{N_T - 1}
    \EE\|\hat{m}^{k+1}_2(t_i,\tilde{X}_{t_i}^{k,\pi}) - m_2(t_i,\tilde{X}_{t_i}^{k,\pi}, \nu_{t_i}^k)\|^2\Delta t_i \notag\\
    & +\EE\|\hat{m}_3^{k+1}(\tilde{X}_T^{k,\pi}) - m_3(\tilde{X}_T^{k,\pi},\mu_T^{k})\|^2\Big\} \notag\\
    & \qquad \qquad \le C\left[\frac{1}{N}+\|\pi\| + \EE\|\tilde{Y}_T^{k,\pi} - g(\tilde{X}_T^{k,\pi}, \hat{m}_3^k(\tilde{X}_T^{k,\pi}))\|^2\right].\label{learn_error_upper}
\end{align}
\end{thm}
\begin{proof}
Combining Assumptions~\ref{function_class_assumption} and \ref{assumption bsde}, the fact that $(\hat{m}_1^k,\hat{m}_2^k,\hat{m}_3^k) \in \mathcal{M}$ and \cite[Theorem~1 and 2]{han2020convergence1}, we obtain \eqref{bsde_error_upper} and
\begin{multline}
     \sup_{0 \le t \le T}[\EE\|\tilde{X}_t^k - \tilde{X}_{\pi(t)}^{k,\pi}\|^2  + \EE\|\tilde{Y}_t^k - \tilde{Y}_{\pi(t)}^{k,\pi}\|^2] + \int_{0}^{T}\EE\|\tilde{Z}_t^k - v_i^k(\tilde{X}_{\pi(t)}^{k,\pi})\|_F^2\rmd t \\
    \le C[\|\pi\| +  \EE\|\tilde{Y}_T^{k,\pi} - g(\tilde{X}_T^{k,\pi}, \hat{m}_3^k(\tilde{X}_T^{k,\pi}))\|^2]\label{bsde_error_lower}.
\end{multline}
Applying Theorem~\ref{thm_iteration} to $(\tilde{X}_t^k,\tilde{Y}_t^k,\tilde{Z}_t^k)$, we have
\begin{multline}
    \sup_{0 \le t \le T}\big[\EE\|X_t - \tilde{X}_t^k\|^2 + \EE\|Y_t - \tilde{Y}_t^k\|^2\big]  + \int_{0}^T\EE\|Z_t - \tilde{Z}_t^k\|_F^2 \rmd t \\
    \le C\Big\{q^k +\sum_{j=0}^{k-1}q^{k-j}\Big[\int_{0}^T\Big(\EE\|\hat{m}_1^{j+1}(t,\tilde{X}_t^{j+1}) - m_1(t,\tilde{X}_t^{j+1},\MCL(\tilde{X}_t^j))\|^2\\\quad\quad\quad\quad\quad\quad\quad\quad\quad\quad\quad\quad\quad\quad+ \EE\|\hat{m}_2^{j+1}(t,\tilde{X}_t^{j+1}) - m_2(t,\tilde{X}_t^{j+1},\MCL(\tilde{X}_t^j,\tilde{Y}_t^j))\|^2\Big)\rmd t \\
    + \EE\|\hat{m}_3^{j+1}(\tilde{X}_T^{j+1}) - m_3(\tilde{X}_T^{j+1},\MCL(\tilde{X}_T^j))\|^2\Big] \Big\}.
\end{multline}
Using Theorem~\ref{thm_maps}, we know that 
\begin{equation}
    (m_1(t,x,\MCL(\tilde{X}_t^k)),m_2(t,x,\MCL(\tilde{X}_t^k.\tilde{Y}_t^k)),m_3(x,\MCL(\tilde{X}_T^k)) \in \mathcal{M}.\label{regularity_m}
\end{equation}
Then Lemma~\ref{lem: Girsanov} gives that
\begin{multline}\label{thm_main:eq1}
    \sup_{0 \le t \le T}\big[\EE\|X_t - \tilde{X}_t^k\|^2 + \EE\|Y_t - \tilde{Y}_t^k\|^2\big]  + \int_{0}^T\EE\|Z_t - \tilde{Z}_t^k\|_F^2 \rmd t\\
    \le C\Big\{q^k +\sum_{j=0}^{k-1}q^{k-j}\Big[\int_{0}^T\Big(\EE\|\hat{m}_1^{j+1}(t,\tilde{X}_t^{j}) - m_1(t,\tilde{X}_t^{j},\MCL(\tilde{X}_t^j))\|^2\\ +\EE\|\hat{m}_2^{j+1}(t,\tilde{X}_t^{j}) - m_2(t,\tilde{X}_t^{j},\MCL(\tilde{X}_t^j,\tilde{Y}_t^j))\|^2\Big)\rmd t 
     \\
     + \EE\|\hat{m}_3^{j}(\tilde{X}_T^{j}) - m_3(\tilde{X}_T^{j},\MCL(\tilde{X}_T^j))\|^2\Big] \Big\}^\epsilon.
\end{multline}
Combining with inequality \eqref{bsde_error_lower} and Theorem \ref{thm: GMMD_error}, we know that for any $k \in \mathbb{N}$,
\begin{multline}
    \EE \rmD_\Phi^2(\MCL(X_t^k), \mu_{\pi(t)}^k) + \EE \rmD_{\Phi'}^2(\MCL(X_t^k,Y_t^k), \nu_{\pi(t)}^k) \\
    \le C\Big[\|\pi\| + \frac{1}{N} + \EE\|\tilde{Y}_T^{k,\pi} - g(\tilde{X}_T^{k,\pi}, \hat{m}_3^k(\tilde{X}_T^{k,\pi}))\|^2\Big].\label{gmmd_error}
\end{multline}
Combining the last inequality with inequalities \eqref{bsde_error_lower} and \eqref{thm_main:eq1}, we have deduced inequality \eqref{thm_main:main}.

Finally, we prove the inequality \eqref{learn_error_upper}. Recalling \eqref{regularity_m}, it suffices to show
\begin{align*}
    &\sum_{i = 0}^{N_T - 1}\Big(
    \EE\|m_1(t_i,\tilde{X}_{t_i}^{k,\pi},\MCL(\tilde{X}_{t_i}^k)) - m_1(t_i,\tilde{X}_{t_i}^{k,\pi}, \mu_{t_i}^k)\|^2 \\
    & \qquad \qquad + \EE\|m_2(t_i,\tilde{X}_{t_i}^{k,\pi},\MCL(\tilde{X}_{t_i}^k,\tilde{Y}_{t_i}^k)) - m_1(t_i,\tilde{X}_{t_i}^{k,\pi}, \nu_{t_i}^k)\|^2\Big)\Delta t_i \\ 
    &  +\EE\|m_3(\tilde{X}_T^{k,\pi},\MCL(\tilde{X}_T^k)) - m_3(\tilde{X}_T^{k,\pi},\mu_T^{k})\|^2 \\
    & \qquad \qquad \qquad 
     \le C\Big[\frac{1}{N}+\|\pi\| + \EE\|\tilde{Y}_T^{k,\pi} - g(\tilde{X}_T^{k,\pi}, \hat{m}_3^k(\tilde{X}_T^{k,\pi}))\|^2\Big],
\end{align*}
which follows inequality \eqref{gmmd_error}. Thus we've obtained inequality \eqref{learn_error_upper}.
\end{proof}

Theorem \ref{main_thm} shows that if we further require $(\mathfrak{m}_1,\mathfrak{m}_2,\mathfrak{m}_3) \in \mathcal{M}$ in the optimization problems \eqref{def:learn_m_nn},
the Deep MV-FBSDE algorithm will converge. Specifically, the inequality \eqref{thm_main:main} shows that the distance between the true solution of the MV-FBSDE \eqref{def:MV-FBSDE-simple} and the output of the Deep MV-FBSDE algorithm can be controlled by the mesh size, the number of samples to approximate the distribution, the loss functions for the supervised learning problems \eqref{def:learn_m_nn} and the Deep BSDE method \eqref{eq:deepBSDE-goal} at all the previous stages and an exponential decay term with respect to the number of iterations. Inequality \eqref{bsde_error_upper} is a direct application of Theorem 2 in \cite{han2020convergence1}, which shows that the loss function of the Deep BSDE method at each stage can be small if the approximation capacity of the parametric function spaces ($\mathcal{N}'$ and $\{\mathcal{N}_i\}_{i=0}^{N_T-1})$ is large. Finally, inequality \eqref{learn_error_upper} shows that the loss function for the supervised learning can be small if the loss function of the Deep BSDE method at the last stage is small. In practice, to ensure that the optimized neural networks $\mathfrak{m}_1, \mathfrak{m}_2, \mathfrak{m}_3$ satisfy the Lipschitz constraints, we can use the method in \cite{fazlyab2019efficient} to obtain an efficient and accurate estimation of Lipschitz constants for neural networks. If needed, we can also use various techniques proposed in recent literature (see, {\it e.g.},~\cite{arjovsky17wgan,gulrajani2017improved,finlay2018lipschitz,fazlyab2019efficient,pauli2021training}) to regularize the Lipschitz constants during the optimization.

\section{Numerical Results}\label{sec:numerics}

The algorithm is implemented in Python using the machine learning library TensorFlow. The code can be found in the public GitHub repository\footnote{\url{https://github.com/frankhan91/DeepMVBSDE}}, and thus the results presented here can be straightforwardly reproduced and further developed.

\subsection{A Benchmark Example with Explicit Solution}\label{sec:numerics-benchmark}

We first consider the following MV-FBSDEs for $(X_t, Y_t, Z_t)$:
\begin{equation}\label{eq:exam1}
{\small    \begin{dcases}
    \mathrm{d} X_t^i = \left[\mathrm{sin}(\tilde\EE_{x_t'\sim\mu_t} e^{-\frac{\|X_t- x_t'\|^2}{d}} - e^{-\frac{\|X_t\|^2}{d+2t}}\Big(\frac{d}{d+2t}\Big)^\frac{d}{2})+ \frac12(m_t^Y - \mathrm{sin}(t) e^{-\frac{t}{2}}) \right]\rmd t + \rmd W_t^i, \\ 
    \mathrm{d} Y_t = \left[\frac{Z_t^1 + \dots + Z_t^d}{\sqrt{d}}- \frac{Y_t}{2}+ \sqrt{Y_t^2 + \|Z_t\|^2 + 1}-\sqrt{2}\right]\rmd t + Z_t\rmd W_t, \\ 
    \end{dcases}}
\end{equation}
with initial and terminal conditions $X_0^i = 0$ and $Y_T = \mathrm{sin}\Big(T + \frac{X_T^1 + \dots+ X_T^d}{\sqrt{d}}\Big)$,
where $X_t^i$ denotes the $i^{th}$ entry of the $d$-dimensional forward process $X_t$, $\mu_t = \mc{L}(X_t)$,  $m_t^Y = \EE[Y_t]$, $p = 1$ and $q= d$. The expectation $\tilde \EE$ is with respect to $x_t' \sim \mu_t$ only.
It can be verified straightforwardly that
\begin{equation*}
    X_t = W_t, \quad  Y_t = \sin\Big(t + \frac{X_t^1 + \dots +X_t^d}{\sqrt{d}}\Big), \quad  Z_t^i = \frac{1}{\sqrt d}\cos\Big(t + \frac{X_t^1 + \dots +X_t^d}{\sqrt{d}}\Big),
\end{equation*}
is the solution of the above MV-FBSDEs. Here the distribution-dependent functions are defined as
\begin{align*}
     &m_1(t, x, \MCL(\Theta_t))) = \tilde \EE_{x'\sim \mc{L}(X_t)} e^{-\frac{\|x- x'\|^2}{d}} + \frac12 m_t^Y, \\
     &m_2 \equiv 0, \quad m_3 \equiv 0.
\end{align*}
We first apply the proposed Deep MV-FBSDE algorithm to solve this problem when the dimension $d$ equals $5$ and $10$, and $T=0.5$. We have tested both the Deep BSDE solver and the DBDP solver for this problem. We discretize the time interval into 20 equal subintervals and simulate $N=1500$ SDE paths to approximate the state distribution $\mathcal{L}(\Theta_t)$. The fictitious play is run for $30$ stages in total. We use feedforward neural networks with two hidden layers of width 24 and Rectified Linear Unit (ReLU) activation function to approximate the unknown functions in the BSDE solver and $\mathfrak{m}_1$ when approximating the distribution dependence. 

\begin{figure}[!htb]
\centering
\subfloat[$d=5$]{\includegraphics[width=0.9\textwidth]{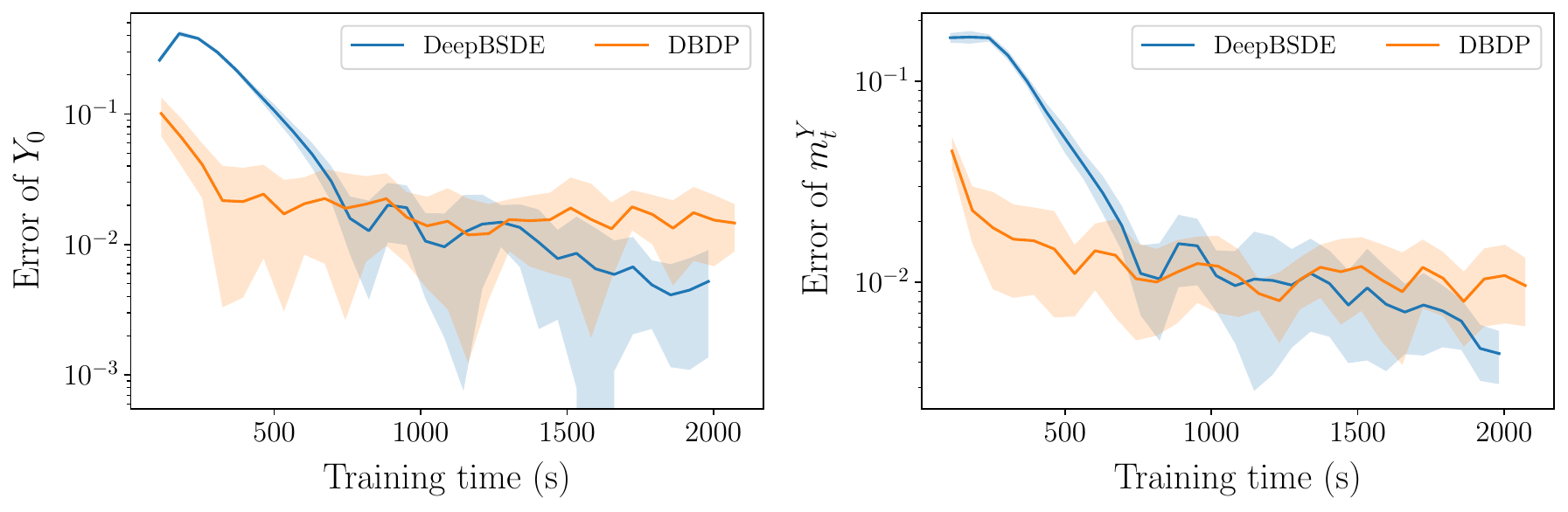}}\\
\subfloat[$d=10$]{\includegraphics[width=0.9\textwidth]{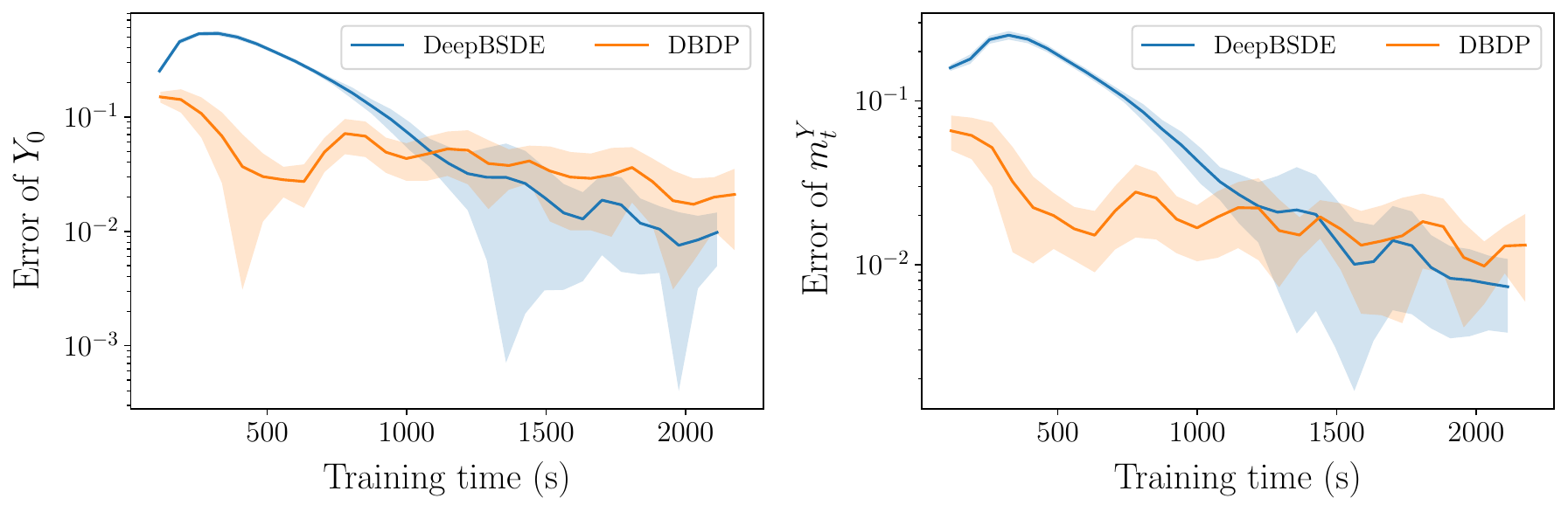}}
  \caption{The absolute error of $Y_0$ (left) and $L^2$ error of $m_t^Y$ (right) for  equation~\eqref{eq:exam1} along the training process, in the case of $d=5$ (top) and $d=10$ (bottom). The shaded area depicts the mean $\pm$ the standard deviation computed from 5 independent runs.
  \label{fig:sinebm_err}
  }
\end{figure}

The training curves of absolute error of $Y_0$ and $L^2$ error of $m_t^Y$ versus training time (implemented on a Macbook Pro with a 2.40 Gigahertz Intel Core i9 processor) are shown in Figure~\ref{fig:sinebm_err}. Using the Deep BSDE method, the algorithm finally finds the solution whose absolute error of $Y_0$ and $L^2$ error of $m_t^Y$ are 0.54\% and 0.43\% when $d=5$ and 1.01\% and 0.80\% when $d=10$, respectively. Using the DBDP method, the algorithm finally finds the solution whose absolute error of $Y_0$ and $L^2$ error of $m_t^Y$ are 1.60\% and 0.92\% when $d=5$ and 1.75\% and 1.09\% when $d=10$, respectively. The results corroborate our convergence analysis that the proposed algorithm can find accurate solutions in high dimensions with either BSDE solver. From Figure~\ref{fig:sinebm_err} and the final accuracy, we can see that the Deep BSDE solver performs slightly better than the DBDP solver when applied to this problem, consistent with the results reported in~\cite{germain2019numerical}.

We further investigate the numerical error of the Deep MV-FBSDE algorithm with varying dimensions $d$, using the Deep BSDE method as a representative BSDE solver. The dimensions considered are $d=5,8,10,12,15$. Two important parameters are increased linearly with $d$. Specifically, the number of SDE paths $N$ is set at (500, 800, 1000, 1200, 1500) for the corresponding values of $d$, and the network widths are set at (12, 18, 24, 30, 36), respectively. In contrast, other settings are kept constant across dimensions: our network architecture consists of feedforward neural networks with two hidden layers and a ReLU activation function, the number of fictitious stages remains fixed at 80, and the learning rate schedule is unchanged.
Figure~\ref{fig:dim_scaling} (the left panel) illustrates the error bars (computed from 5 independent runs) for the final error of $Y_0$ across different dimensions. For the problem~\eqref{eq:exam1}, as the dimension increases from 5 to 15, we maintain an approximate solution for $Y_0$ with an absolute error below $2\%$. This is accomplished with a reasonably modest increase in computational cost. Figure~\ref{fig:dim_scaling} (the right panel) depicts the evolution of the error during fictitious play. As expected, convergence requires more fictitious plays as the dimension increases, but the increase is not dramatic.
These results suggest that the proposed Deep MV-FBSDE algorithm exhibits favorable scaling with respect to problem dimensions, standing in stark contrast to traditional numerical methods that are susceptible to the curse of dimensionality.

\begin{figure}[!htb]
\centering
{\includegraphics[width=0.9\textwidth]{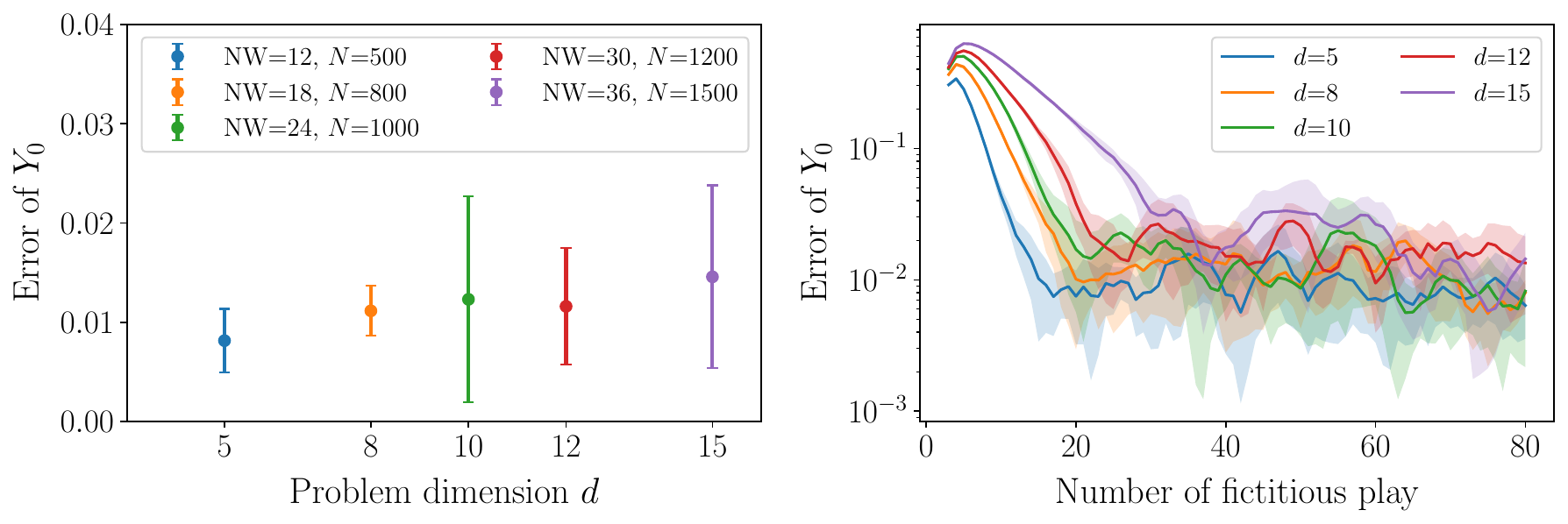}}\\
  \caption{The absolute error of $Y_0$ in solving equation~\eqref{eq:exam1} using the Deep BSDE method across different dimensions $d=(5,8,10,12,15)$. The network width (NW) and the number of particles $N$ increase linearly as $d$ increases.  The left panel shows error bars for the final errors, while the right panel portrays the error evolution during fictitious play. The shaded area depicts the mean $\pm$ the standard deviation computed from 5 independent runs.
  \label{fig:dim_scaling}
  }
\end{figure}

\subsection{A Mean-field Game of Cucker-Smale Flocking Model}
In this subsection, we consider a mean-field problem of Cucker-Smale \cite{cucker2007mathematics} flocking model motivated from \cite{nourian2010synthesis,nourian2011mean}. The Cucker-Smale model plays an important role in modeling and analyzing the flocking behaviors, which refers to the collective motion by a group of self-propelled entities such as birds, fish, bacteria, and insects. When the number of entities is very large, it is natural to consider the mean-field limit system. Below, we consider a mean-field game version of the Cucker-Smale model where every agent chooses her acceleration to minimize the cost of acceleration and their location and velocity misalignments.

We model a representative agent's dynamics by
\begin{equation*}
    \begin{dcases}
      \ud x_t = v_t \ud t, \\
      \ud v_t = u_t \ud t + C \ud W_t,
    \end{dcases}
\end{equation*}
where $x_t \in \RR^n$ and $v_t \in \RR^n$ are time-$t$ position and velocity vectors of the agent, $u_t \in \RR^n$ is the acceleration control input, and $W$ is a $p$-dimensional Brownian motion. The agent aims to minimize
\begin{equation}\label{eq:CSflock_cost}
     \EE \int_0^T \|u_t\|_R^2 +  \mathcal{C}(x_t, v_t; f_t)  \ud t,
\end{equation}
over possible $u_t$, where the first term is the cost of acceleration, and the second term $\mathcal{C}(x_t, v_t; f_t)$ describes the agent's position and velocity's misalignment from a given distribution $f_t$:
\begin{multline*}
    \mathcal{C}(x,v; f) = \left\Vert\int_{\RR^{2n}} w(\|x - x'\|)(v' - v) f(x', v') \ud x' \ud v'\right\Vert_Q^2 \\
    = \left\Vert \tilde \EE_{(x',v')\sim f} [w(\|x - x'\|)(v' - v)]\right\Vert_Q^2.
\end{multline*}
Here $w(x):= \frac{1}{(1+ x^2)^\beta}\,(\beta\geq0)$, the expectation $\tilde \EE$ is with respect to the distribution $f$, and $Q, R$ are symmetric positive definite matrices with compatible dimensions such that $\|x\|_Q := (x\transpose Qx)^{1/2}$. In order to find the equilibrium of the mean-field game, one needs to find the optimal $\hat u_t$ such that $f_t$ is the density of the associated optimal position and velocity $(\hat x_t, \hat v_t)$. We remark that our setting is a stochastic MFG on the finite horizon $[0,T]$, while in \cite{nourian2010synthesis,nourian2011mean}, the authors formulate it as an ergodic mean-field control problem. The infinite horizon MFG was studied in \cite{perrin2021mean} via reinforcement deep learning, deterministic MFGs of Cucker-Smale model on finite horizon was also studied in the past \cite{santambrogio2021cucker,bardi2021convergence}.

\subsubsection{The Reformulation through MV-FBSDEs}

We can use the stochastic maximum principle \cite[Section~3.3.2]{CaDe1:17} to characterize the mean-field equilibrium through MV-FBSDEs. To this end, let us define the Hamiltonian $H$ by
\begin{equation*}
    H(t, x, v, f, y, u) = (v\transpose, u\transpose) y +\mathcal{C}(x, v; f)  + \|u\|_R^2,
\end{equation*}
where $y \in \RR^{2n}$. Denote by $y_{1:n}$ the first half entries of this vector and by $y_{n+1:2n}$ the second half, then the unique minimizer of $H$ is given by:
\begin{equation*}
    \hat u = -\half R\inv y_{n+1:2n}. 
\end{equation*}
With straightforward computation, the derivatives of $H$ with respect to the state variables $(x, v)$ are:
\begin{align*}
    \partial_x H & = \partial_x \mathcal{C}(x, v; f) \\
    & = 2\tilde \EE_{(x',v') \sim f}[\partial_xw(\|x - x'\|)(v' - v)]\transpose Q \tilde \EE_{(x',v')\sim f} [w(\|x - x'\|)(v' - v)], \\
    \partial_v H & = y_{1:n} + \partial_v \mathcal{C}(x, v; f) \\
    & = y_{1:n} + 2Q \tilde \EE_{(x',v')\sim f} [w(\|x - x'\|)(v' - v)] \tilde \EE_{(x',v') \sim f}[-w(\|x - x'\|)],
\end{align*}
where $\partial_xw(\|x - x'\|)(v' - v)$ is understood as a Jacobian matrix
$$
\begin{bmatrix}
 \frac{\partial w(\|x-x'\|)}{\partial x_1}(v_1' - v_1) & \cdots &  \frac{\partial w(\|x-x'\|)}{\partial x_n}(v_1' - v_1) \\
\vdots & \cdots & \vdots \\
 \frac{\partial w(\|x-x'\|)}{\partial x_1}(v_n' - v_n) & \cdots &  \frac{\partial w(\|x-x'\|)}{\partial x_n}(v_n' - v_n)
\end{bmatrix},
$$
and $x_i$, $v_i$ denotes the $i^{th}$ entry of $x$, $v$, respectively.

By the stochastic maximum principle, $(\hat u_t, \hat f_t)_{0 \leq t \leq T}$ is an MFG equilibrium if and only if $\hat u_t = -\half R\inv Y_t^2$, $\hat f_t = \mc{L}( x_t, v_t)$ and $(x_t, v_t,  Y_t,  Z_t)$ solves the MV-FBSDEs
\begin{equation}
\label{eq:CSFlocking}
\begin{dcases}
    \ud x_t = v_t \ud t, \quad \ud v_t = -\half R\inv Y_t^2 \ud t + C \ud W_t, \quad &(x_0, v_0) = \xi,  \\
   \ud Y_t = -\begin{pmatrix}\partial_x H\\ \partial_v H \end{pmatrix}(t, x_t, v_t, \mc{L}(x_t, v_t), Y_t, \hat u_t) \ud t + Z_t \ud W_t, \quad &Y_T = 0,
    \end{dcases}
\end{equation}
where $Y_t \in \RR^{2n}$ is the backward process with $Y_t^1 \in \RR^n, Y_t^2 \in \RR^n$ be the first half and second half entries,  and $Z_t = \begin{pmatrix}Z_t^1 \\ Z_t^2\end{pmatrix} \in \RR^{2n\times p}$ is the adjoint process with $Z_t^1$ and $Z_t^2$ being $\RR^{n\times p}$-valued. Therefore the distribution dependence functions are identified as
\begin{align*}
&m_1 \equiv 0, \quad m_3 \equiv 0,\\
& m_2(t, x_t, v_t, \mc{L}(x_t, v_t)) = \left[
\begin{aligned}
&\tilde\EE[\partial_xw(\|x_t - x'_t\|)(v'_t - v_t)]\transpose Q \tilde \EE [w(\|x_t - x'_t\|)(v'_t - v_t)]\\
&\tilde \EE [w(\|x_t - x'_t\|)(v'_t - v_t)] \tilde \EE[-w(\|x_t - x'_t\|)]
\end{aligned}
\right],
\end{align*}
where $(x_t', v_t')$ is a copy of $(x_t, v_t)$ and $\tilde \EE$ is with respect to $(x_t', v_t')$ distributed according to $\mc{L}(x_t, v_t)$.

The system~\eqref{eq:CSFlocking} hardly admits an explicit solution in general.
A special case is when $\beta= 0$, $w(x) \equiv 1$, and the problem reduces to a linear-quadratic mean-field game in which the mean-field cost-coupling function becomes
\begin{equation*}
    \mathcal{C}(x,v; f) = \left\Vert \tilde \EE_{(x',v')\sim f} [v' - v]\right\Vert_Q^2 = \left\Vert\EE[v] - v\right\Vert_Q^2.
\end{equation*}
The derivatives of Hamiltonian are then simplified to
\begin{equation*}
    \partial_x H = 0, \quad \partial_v H = y_{1:n} + 2Q(v - \tilde \EE_{(x', v')\sim f}[v']) = y_{1:n} + 2Q(v - \EE[v]).
\end{equation*}
Now the FBSDEs derived from the stochastic maximum principle become
\begin{equation}\label{eq:CSFlocking_special}
\begin{dcases}
    \ud x_t = v_t \ud t, \quad  \ud v_t = -\half R\inv Y_t^2 \ud t + C \ud W_t, \quad &(x_0, v_0) = \xi,  \\
   \ud Y_t = -\begin{pmatrix}0\\ Y_t^1 + 2Q(v_t - \EE[v_t]) \end{pmatrix} \ud t + Z_t \ud W_t, \quad & Y_T = 0,
    \end{dcases}
\end{equation}
where one immediately deduces $Y_t^1 \equiv Z_t^1 \equiv 0$. We then only need to solve the forward-backward system of $(v_t, Y_t^2)$ only. Let the ansatz of $Y_t^2$ be 
\begin{equation*}
    Y_t^2 = \eta(t) v_t + \chi(t), \quad \eta(T) = \chi(T) = 0,
\end{equation*}
where $\eta(t) \in \RR^{n \times n}$ and $\chi(t) \in \RR^n$. Plugging it back to \eqref{eq:CSFlocking_special} gives the following system:
\begin{equation*}
    \begin{dcases}
    \dot\eta(t) - \half \eta(t) R\inv \eta(t) + 2Q = 0, \quad \eta(T) = 0, \\
    \dot \chi(t) - \half \eta(t) R \inv \chi(t) - 2Q\EE[v_0] = 0, \quad \chi(T) = 0.
    \end{dcases}
\end{equation*}
One can also deduce that $Z_t^2 = \eta(t)C$, $\EE[Y_t^2] \equiv 0$ and $\EE[v_t] \equiv \EE[v_0]$. The solution to the above analytical matrix Riccati equation serves as the benchmark solution to the experiment below when $\beta=0$.

\subsubsection{Numerical Results}
We use the proposed method to solve the problem numerically with $n=p = 3$, and choose $Q=I_3$, $R=0.5 I_3$, $C=0.1 I_3$, $T=1$, and $\beta=0$, $0.1$, $0.2$, $0.3$, where $I_3$ denotes the identity matrix of size 3. Thus, the solution to \eqref{eq:CSFlocking} is of dimension $\RR^6 \times \RR^6 \times \RR^{6\times 3}$. We assume $x_0 \sim \mathcal{N}(0, I_3)$, $v_0 \sim \mathcal{N}(1, I_3)$. We discretize the time interval $[0, 1]$ into 50 equal subintervals and simulate $N=1000$ SDE paths to approximate the state distribution $\mathcal{L}(x_t, v_t)$. We use feedforward neural networks with two hidden layers of width 48 and Rectified Linear Unit (ReLU) activation function to approximate the unknown functions in the BSDE solver and $\mathfrak{m}_2$ when approximating the distribution dependence. The fictitious play is run for $60$ stages and the total runtime is 474 seconds (implemented on a Macbook Pro with a 2.40 Gigahertz Intel Core i9 processor). The numerical solution of $\beta = 0$ is very close to the analytical solution, with an $L^2$ error of $Y_0^2$ being $0.017$. 

In the left and middle panels of Figure~\ref{fig:flock_state}, we plot the density of the final position $x_T$ and the final velocity $v_T$. Since we choose the parameter such that the problem is symmetric with three dimensions, we only present the first entry of $x_T$ and $v_T$. In the right panel of Figure~\ref{fig:flock_state}, we present the standard deviation of $v_t$'s first entry as a function of time $t$ with different $\beta$. From all three subplots in Figure~\ref{fig:flock_state} we can see that, when $\beta=0$, the analytical solution and our numerical solution agree very well with each other. 
We also observe that the standard deviation of $v$ increases as $\beta$ increases. This qualitative trend is consistent with our intuition since the cost of ``misalignment" of the location-velocity decreases as $\beta$ increases. This intuition is further confirmed in Figure~\ref{fig:flock_control}, in which we plot the density of the control $u_t$'s first entry at $t=0.0$, $0.3$, $0.6$. It is mathematically clear that as $\beta$ increases, the second cost term $\mathcal{C}(x_t, v_t; f_t)$ in equation~\eqref{eq:CSflock_cost} punishes the location-velocity ``misalignment" less severe.  Thus the agents will not spend as much effort (large acceleration $u_t$) as they would for smaller $\beta$, as they also want to minimize first cost term $\|u\|^2_R$ in \eqref{eq:CSflock_cost}. As a result, the density of $u_t$ is more centered at $0$ for large $\beta$.

\begin{figure}[!htb]
\centering
{\includegraphics[width=1.0\textwidth]{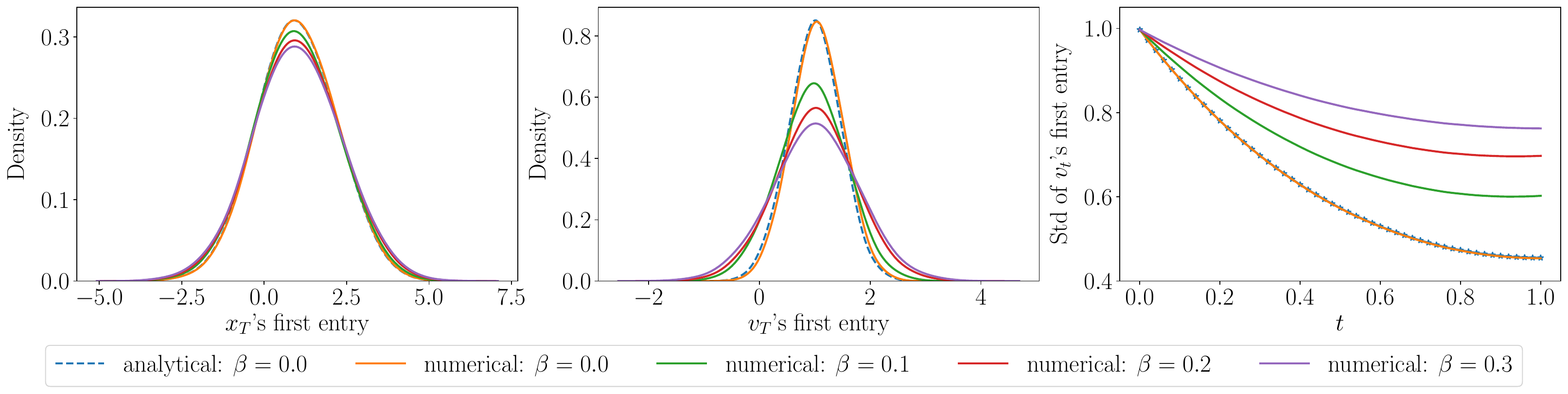}}
  \caption{Numerical solution of the states to the Cucker-Smale MFG~\eqref{eq:CSFlocking}. Left: the density of the position $x_T$'s first entry with different $\beta$. Middle: the density of the velocity $v_T$'s first entry with different $\beta$. Right: the standard deviation of $v_t$'s first entry as a function of time with different $\beta$. 
  \label{fig:flock_state}
  }
\end{figure}

\begin{figure}[!htb]
\centering
{\includegraphics[width=1.0\textwidth]{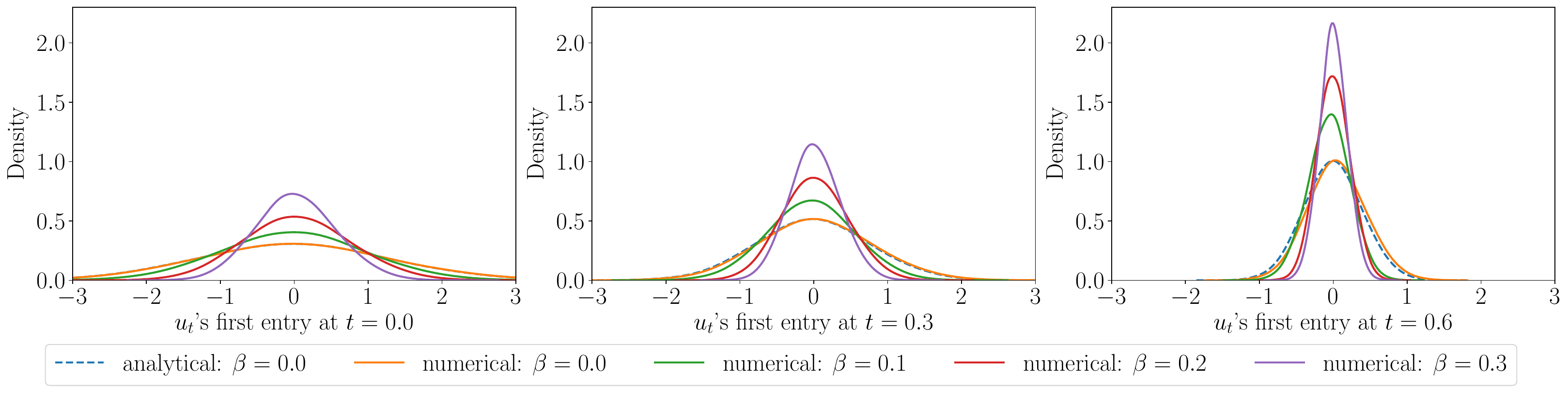}}
  \caption{Numerical solution of the optimal control to the Cucker-Smale MFG~\eqref{eq:CSFlocking}. From left to right: the density of $u_t$'s first entry with different $\beta$ at different time $t=0.0$, $0.3$, $0.6$.
  \label{fig:flock_control}
  }
\end{figure}

\section{Conclusion}\label{sec:conclusion}
Built on deep learning BSDE solvers, we propose a novel deep learning method for solving the McKean-Vlasov forward-backward stochastic differential equations, which have been frequently used in mean-field control problems and mean-field games. Specifically, we can compute the general cases when the mean-field interactions depend on full distributions, not only on expectation or other moments. We prove that the convergence of the proposed method is free of the curse of dimensionality by using a class of integral probability metrics \cite{han2021class}. The proved theorem shows the advantage of the method in high dimensions. We then derived a mean-field game of the well-known Cucker-Smale model, with the cost depending on the full distribution of the forward process. The numerical results are nicely in line with the model's intuition. In the future, we plan to study numerical algorithms for MV-FBSDEs where the dependence of $\mc{L}(X_t)$ is not by the integral form, but via $f(X_t)$ where $f(\cdot)$ is the density or cumulative distribution function of the forward process $X_t$. This is motivated by the famous beach towel mean-field game \cite{lions2007theorie}. Another plan is to extend the algorithm to MV-FBSDEs in a random environment, corresponding to mean-field control and game problems with a common noise. 

\section*{Acknowledgments}
R.H. was partially supported by the NSF grant DMS-1953035, the Faculty Career Development Award, the Research Assistance Program Award, the Early Career Faculty Acceleration funding and the Regents' Junior Faculty Fellowship at the University of California, Santa Barbara, and a grant from the Simons Foundation (MP-TSM-00002783).

\bibliographystyle{plain}
\bibliography{Reference}

\end{document}